\documentclass[10pt]{article}  
\usepackage{color}              
\usepackage{graphicx}
\usepackage[]{amsmath}
\usepackage{amssymb}
\usepackage{hyperref,bbm}
\definecolor{MyDarkBlue}{rgb}{0,0.08,0.50}
\definecolor{BrickRed}{rgb}{0.65,0.08,0}

\hypersetup{
colorlinks=true,       
    linkcolor=MyDarkBlue,          
    citecolor=BrickRed,        
    filecolor=red,      
    urlcolor=cyan           
}

\usepackage{enumerate}
\usepackage[mathscr]{eucal}
\usepackage{graphicx}
\usepackage{latexsym}
\usepackage{amssymb}
\usepackage{psfrag}
\usepackage{epsfig}
\usepackage{tikz}
\usetikzlibrary{automata,positioning}

\include{psbox}

\usepackage{amsfonts}

\usepackage{graphicx}

\usepackage{amsfonts}

\usepackage{color}

\newtheorem{Theorem}{Theorem}
\newtheorem{Assumption}{Assumption}
\newtheorem{Lemma}{Lemma}[section]
\newtheorem{lemma}{Lemma}[section]
\newtheorem{Proposition}[Lemma]{Proposition}

\newtheorem{Corollary}{Corollary}

\newtheorem{remark}[Lemma]{Remark}

\usepackage{accents}

\newcommand{\indep}{\perp \!\!\! \perp}
\newcommand{\normal}{\operatorname{N}}
\newcommand{\exponential}{\operatorname{Exp}}

\newcommand{\s}{\quad}

\newcommand{\R}{\mathbb{R}}

\newcommand{\non}{\nonumber}

\newcommand{\beq}{\begin{eqnarray*}}
\newcommand{\eeq}{\end{eqnarray*}}
\newcommand{\beqn}{\begin{eqnarray}}
\newcommand{\eeqn}{\end{eqnarray}}

\newcommand{\bt}{\begin{Theorem}}
\newcommand{\et}{\end{Theorem}}

\newcommand{\bas}{\begin{Assumption}}
\newcommand{\eas}{\end{Assumption}}

\newcommand{\be}{\begin{equation}}
\newcommand{\ee}{\end{equation}}

\newcommand{\N}{\mathbb{N}}

\newcommand{\mb}[1]{{\mathbf #1}}

\newcommand{\Var}{{\rm Var}}

\numberwithin{equation}{section}

\definecolor{darkgreen}{rgb}{0,.4,0}
\definecolor{darkagenta}{rgb}{.5,0,.5}
\definecolor{darkred}{rgb}{1,0,0}
\definecolor{darkblue}{rgb}{0,0,.4}

\begin{document}

\tikzset{every node/.style={auto}}
 \tikzset{every state/.style={rectangle, minimum size=0pt, draw=none, font=\normalsize}}
 \tikzset{bend angle=20}

	\author{Filip Lindskog
	\thanks{Stockholm University, Department of Mathematics, lindskog@math.su.se
	}
\and Abhishek Pal Majumder
	\thanks{Corresponding author, Stockholm University, Department of Mathematics, majumder@math.su.se
	}
      }

\title{Exact long time behavior of some regime switching stochastic processes}

\maketitle

\begin{abstract}
Regime switching processes have proved to be indispensable in the modeling of various phenomena, allowing model parameters that traditionally were considered to be constant to fluctuate in a Markovian manner in line with empirical findings.  
We study diffusion processes of Ornstein-Uhlenbeck type where the drift and diffusion coefficients $a$ and $b$ are functions of a Markov process with a stationary distribution $\pi$ on a countable state space. 
Exact long time behavior is determined for the three regimes corresponding to the expected drift: $E_{\pi}a(\cdot)>0,=0,<0$, respectively. Alongside we provide exact time limit results for integrals of form $\int_{0}^{t}b^{2}(X_{s})e^{-2\int_{s}^{t}a(X_{r})dr}ds$ for the three different regimes. Finally, we demonstrate natural applications of the findings in terms of Cox-Ingersoll-Ross diffusion and deterministic SIS epidemic models in Markovian environments. Exact long time behaviors are naturally expressed in terms of solutions to the well-studied fixed-point equation in law $X\stackrel{d}{=}AX+B$ with $X \indep (A,B)$. 
\end{abstract}

\section{Introduction}
Models based on regime switching stochastic processes have received considerable attention for their applications in
quantitative finance, actuarial science, economics, biology and ecology. 
In quantitative finance, volatility, interest rates and asset prices are subjects to risky market environments that fluctuate over different regimes in a Markovian manner. Understanding how critical parameters (that determine stability or instability of the process of interest) characterizing the ``switching regimes" vary stochastically over time and affect the long time behavior of the overall process is essential for making short and long term predictions. 
Examples of such applications are 
\cite{AngTimmermann12}, \cite{BenSaida15}, \cite{ fink2017regime} and \cite{GCJL00} in the context of stochastic volatility modelling in financial market;  
\cite{zhang2016long} considering stochastic interest rate models with Markov switching;  
\cite{Hardy01}, \cite{ LinKenHailiang09} and \cite{ShenSiu13} studying long term behavior of stock returns and bond pricing.  
Similar to quantitative finance, regime switching stochastic processes are frequently used in actuarial science for solvency investigations, e.g. \cite{abourashchi2016pension}, mortality modeling, e.g. \cite{GMLT15}, and in the context of disability insurance, e.g. \cite{DjehicheLofdahl18}. 

Monographs containing both the theoretical foundations and applications of regime switching processes are \cite{mao2006stochastic} and \cite{yin2010hybrid}. Significant contributions to the theoretical foundation are 
\cite{shao2014ergodicity}, \cite{shao2015ergodicity} and \cite{jinghai2015criteria} by Jinghai Shao. 
A common theme of these works is a stochastic dynamical system $(Y_{t},X_{t})_{t\ge 0}$, where the process of interest $Y:=(Y_{t})_{t\ge 0}$ is affected by the process $X:=(X_{t})_{t\ge 0}$ that describes the dynamics of a switching environment. 
For a class of general diffusion processes $Y$ the aforementioned works investigated necessary and sufficient conditions under which properties related to stability/instability such as geometric/polynomial ergodicity \cite{shao2014ergodicity},\cite{shao2015ergodicity}, positive/null recurrence or transience \cite{jinghai2015criteria}, explosivity, existence and uniqueness of moments of stationary distributions hold.  
In a similar context \cite{benaim2016lotka} (and references therein) addresses questions related with survival or extinction of competing species in Lotka-Volterra model influenced by switching parameters in terms of the underlying hidden Markov environment. A main theme is the analysis of persistence (see section 4 or Theorem 4.1 of \cite{benaim2016lotka}) phrased in terms of so-called Lyapunov drift type criteria and similar concepts. 
In \cite{cloez2015exponential} a large class of general regime switching Markov processes are considered where a type of condition referred to as ``geometric contractivity" ensures exponential stability of the overall process. 
In contrast to the general stability results described above, there are very few works giving exact characterizations of long time behaviors, which are inevitably model specific. In this paper, we analyze the long time behavior of processes of Ornstein-Uhlenbeck type and Cox-Ingersoll-Ross models in a regime switching context, and provide exact explicit characterizations. To our knowledge, such explicit characterizations have not appeared in the literature.

The initial object of study in this paper is an $\R$-valued Ornstein-Uhlenbeck process in a Markovian environment, denoted by $Y=(Y_{t})_{t\ge 0}$, defined as the solution to the SDE 
\begin{align}
dY_{t}=-a(X_{t})Y_{t}dt + b(X_{t})dW_{t}, \s Y_{0}=y_{0}\in{\R} ,\label{model1}
\end{align}
where $(W_{t})_{t\ge 0}$ is standard Brownian motion which is independent of $X:=(X_{t})_{t\ge 0}$ that represents the background environment. $X$ is an $S$-valued, where $S$ is a countable set, jump type process with rate functions $\lambda_{ij}:\R\to\R_+$, $(i,j)\in S^2$, satisfying
\beqn
P\big[X_{t+\delta}=j\mid X_{t}=i,Y_{t}=x\big]=
\left\{\begin{array}{ll}
\lambda_{ij}(x)\delta+ o(\delta), & \text{ if } i\neq j, \\
1+\lambda_{ii}(x)\delta +o(\delta), & \text{ if } i=j,
\end{array}\right.
\label{model2}
\eeqn
with notation $\lambda_{ii}(x):=-\sum_{j\neq i \in S}\lambda_{ij}(x)$.  The functions $a,b:S\to \R$ are arbitrary (as long as a path-wise unique weak solution of \eqref{model1} can be insured) denoting, respectively, the drift and the diffusion functions. $(X_{t},Y_{t})_{t\ge 0}$ is a Markov process with respect to its natural filtration $(\mathcal{F}^{X,Y}_t)_{t\geq 0}$.
Given ``$\lambda_{ij}(x)$ is constant with respect to $x$", the process $X$ is a continuous-time Markov chain, with respect to its natural filtration $(\mathcal{F}^{X}_t)_{t\geq 0}$, satisfying the hidden Markovian assumption:
\begin{align}
X_{t}\indep_{X_{s}} Y_{s}\quad \text{for all } t>s \label{causal},
\end{align}
saying that for all $t>s$, $X_{t}$ is conditionally independent of $Y_{s}$ given $X_{s}$.

For any arbitrary stochastic process $X := (X_{t})_{t\ge 0}$, by ergodicity we mean that there exists a probability measure $\mu$ such that, regardless of $X_{0}$, the distribution of $X_{t}$ converges weakly to $\mu$ as $t\to\infty$.
$\mu$ is called the limiting measure. If $X$ is Markovian and irreducible in a countable state space $S$, the limiting measure $\mu$ is its unique invariant measure \cite{norris1998markov}.

Throughout the text we assume that the hidden Markov chain $X$ is ergodic with stationary distribution $\pi:=\{\pi_{j}:j\in S\}.$  The process $Y$ is attractive or stable if $E_{\pi}a(\cdot)>0$ holds, otherwise it is divergent if $E_{\pi}a(\cdot)<0$ and null recurrent if $E_{\pi}a(\cdot)=0$ (can be shown using the Lyapunov function construction ideas from \cite{hairer2010convergence}, \cite{jinghai2015criteria}). Under the stability assumption $E_{\pi}a(\cdot)>0$, a trichotomy of possible tail bahaviors of the stationary distribution was established in \cite{bardet2010long}. 
Using Fourier analysis techniques, \cite{ZW2017stationayOU} provided precise results on ergodicity when $|S|=2$. 
A key contribution of our paper is a precise result on ergodicity including an explicit representation for the stationary distribution when $S$ is a countable state space, allowing computation of probabilities $\lim_{t\to\infty}P[Y_{t}>y]$. 
The result is generalized by generalizing the model \eqref{model1} in different ways.  
Under the instability assumption $E_{\pi}a(\cdot)\le 0$ no stationary distribution exists and we determine how $Y$ diverges by providing weak limits for the scaled fluctuations $\frac{\log|Y_{t}|}{\sqrt{t}}$ which translates to the behavior of $|Y_{t}|^{\frac{1}{\sqrt{t}}}$. In all long time results we describe how introducing a regime switching component leads to mixture type representations characterizing the long time behavior.

In parallel to the characterization of the long time behavior of the model \eqref{model1} we provide the corresponding explicit characterization of the long time behavior of integrals of the type
\begin{align}
\int_{0}^{t}d(X_{s})e^{-\int_{s}^{t}c(X_{r})dr}ds\label{Inte1}
\end{align}
for the three regimes $E_{\pi}c>0,=0,<0$ corresponding to positive recurrence, null recurrence and transience. 

Several previous works, 
e.g. \cite{gjessing1997present}, \cite{bertoin2005exponential}, \cite{maulik2006tail}, \cite{behme2015exponential}, \cite{zhang2016long} and \cite{feng2019exponential} have studied exponential functionals of L{\'e}vy processes. For instance, in \cite{gjessing1997present} the asymptotic behavior of integrals of the form $\int_{0}^{t} e^{-R_{s}}dP_{s}$ as $t\to\infty$ was explored, where $P$ and $R$ are independent L{\'e}vy processes and the property 
$$
\big(P_{t},R_{t}\big)\stackrel{d}{=}\big(P_{s},R_{s}\big)+\big(\widetilde{P}_{t-s},\widetilde{R}_{t-s}\big),
\quad\big(\tilde{P},\tilde{R}\big)\quad\text{is an independent copy of }\big(P,R\big),
$$ 
helps the analysis significantly. We derive exact limit results for integrals of the type \eqref{Inte1},  
where $X$ is a continuous time Markov chain on a countable state space. The asymptotic analysis requires quite different methods from those used for the corresponding analysis for exponential functionals of L{\'e}vy processes. 
Asymptotic analysis similar to the one presented in the current paper was done in \cite{bardet2010long}. Proposition 4.1 in \cite{bardet2010long} yields asymptotic bounds for \eqref{Inte1}, but not the exact asymptotic behavior that we present here.
   
The paper is organized as follows. Section \ref{sec_prel} sets notation and presents basic model assumptions. 
In Section \ref{sec_stable}, exact long time characterizations for the stochastic process \eqref{model1} and different generalizations are presented under assumptions corresponding to the stable regime of the aforementioned process \eqref{model1}. Section \ref{sec_unstable} presents long time characterizations corresponding to the unstable regime when no stationary distribution exists. Section \ref{sec_applications} contains applications of the findings in Sections \ref{sec_stable} and \ref{sec_unstable} to the CIR model, originally introduced as a model for interest rates, and to SIS models used in epidemiology. The proofs are found in Section \ref{sec_proofs}.
  
\section{Preliminaries and model assumptions}\label{sec_prel}

Whenever relevant, random elements appearing are assumed to be defined on a common probability space with probability measure $P$ and expectation operator $E$.
The following notations will be used in this article. $\mathbb{R}^{d}$ will denote the $d$ dimensional Euclidean space with the usual Euclidean norm $|\cdot|$. The set of natural numbers is denoted by $\mathbb{N}$. 
Cardinality of a finite set $S$ is denoted by $|S|$.  For any given sequence $\{a_{n}\}_{n\ge 1}$ define $\{a^{\max}_{n}\}_{n\ge 1}$ as the sequence of running maxima $a^{\max}_{n}:= \max_{1\le k\le n} a_{k}$. 

For a Polish space $S$, let $\mathcal{B}(S)$ be its Borel $\sigma$-field and let $\mathcal{P}(S)$ denote the class of probability measures on $S$.
$\mathcal{P}(S)$ is equipped with the topology of weak convergence. For $x\in S$, $\delta_{x}\in \mathcal{P}(S)$ denotes the Dirac measure that puts unit mass at $x$. 
The probability distribution of an $S$-valued random variable $X$ will be denoted as $\mathcal{L}(X)$. $X\sim \mu$ means that $\mu\in\mathcal{P}(S)$ and $\mu=\mathcal{L}(X)$.
Convergence in distribution of an $S$-valued sequence $(X_{n})_{n\geq 1}$ to an $S$-valued random variable $X$ will be written as $X_{n}\stackrel{d}{\to} X,$ or $\mathcal{L}(X_{n})\stackrel{w}{\to} \mathcal{L}(X)$, where $w$ stands for weak convergence. 

The transition kernels of a Markov process are defined as the maps $P_{s,t} : (S,\mathcal{B}(S))\to [0,1]$ such that for all $t\ge s\ge 0,$ $P_{s,t}(\cdot,A)$ is $\mathcal{B}(S)$-measurable for each $A\in \mathcal{B}(S)$ and $P_{s,t}(i,\cdot)\in \mathcal{P}(S)$ for each $i\in S$. 
The distribution of the Markov process is determined by the transition kernels $P_{s,t}$ together with the initial distribution $\nu_0$. The marginal distribution of the Markov process at time $t$ is $\nu_0 P_{0,t}(\cdot)=\int_S P_{0,t}(x,\cdot)\nu_0(dx)$. 
We will consider only time-homogeneous Markov processes corresponding to transition kernels satisfying $P_{s,t}=P_{0,t-s}$ and use the notation $P_t:=P_{0,t}$.  
$P_{t}f(\cdot)$ is the corresponding transition operator given by $P_{t}f(x)=\int_S f(y)P_{t}(x,dy)$ for functions $f:(S,\mathcal{B}(S))\to (\mathbb{R},\mathcal{B}(\mathbb{R}))$.
A time-homogeneous Markov process in a countable state space $S$ is irreducible if for any $i,j\in S,$ $P_{t}(i,j):=P_{t}(i,\{j\})>0$ for some $t>0.$ A time-homogeneous Markov process with transition kernels $P_{t}$ has a stationary distribution (or invariant law) $\mu \in \mathcal{P}(S)$ if $\mu P_{t} = \mu$ holds for all $t>0$. 

Consider a probability measure $\pi$ on a countable set $S$ such that $\pi(A)=\sum_{i\in A}\pi_{i}$ for any $A\in \mathcal{B}(S)$ and a set of probability measures $\{\mu_j:j\in S\}$. Then  
$$
\sum_{j\in S} \delta_{U}(\{j\})Z_{j}, \quad U \indep (Z_j)_{j\in S}, \quad U\sim \pi, \quad Z_j\sim \mu_j,
$$
is a random variable whose distribution is the mixture distribution $\sum_{j\in S} \pi_{j}\mu_{j}$.
For a given bivariate random variable $(A,B)$, the following time series is referred to as a stochastic recurrence equation (in short SRE, also referred to as random coefficient AR$(1)$) 
\beqn
Z_{n+1}=A_{n+1}Z_{n}+B_{n+1},\s\text{ with }\s (A_{i},B_{i})\stackrel{\text{i.i.d}}{\sim} \mathcal{L}(A,B),\s \,\, Z_{n}\indep (A_{n+1},B_{n+1})\label{stochrec}
\eeqn 
for an arbitrary initial value $Z_{0}=z_{0}\in \R$. Let $\log^{+}|a|:=\log(\max(|a|,1))$.  
If 
\beqn
P[A=0]=0, \quad E\log|A| <0,\s\text{and}\s E\log^{+}|B|<\infty, \label{sre_conv_conds}
\eeqn
then $(Z_{n})$ has a unique causal ergodic strictly stationary solution solving the following fixed-point equation in law: 
\beqn
Z\mathop{=}^{d}AZ+B\s\text{ with} \s Z\indep (A,B).\label{x=ax+b}
\eeqn
The condition $P[Ax+B=x]<1,$ for all $x\in \mathbb{R}$, rules out degenerate solutions $Z=x$ a.s. We refer to Corollary 2.1.2 and Theorem 2.1.3 in \cite{buraczewski2016stochastic} for further details.   

We denote by $\normal(\mu,\sigma^2)$ and $\exponential(\lambda)$, respectively, the Normal distribution with mean $\mu$ and variance $\sigma^2$ and the Exponential distribution with mean $1/\lambda>0$.

Throughout the rest of this paper we will assume the following:

\bas\label{as1} $S$ is a countable set and the $S$-valued Markov process $X:=(X_{t})_{t\ge 0}$ satisfies \eqref{model2}.
\begin{enumerate}[(a)]
\item For every $(i,j) \in S^2$, the rate function $\lambda_{ij}(\cdot)$ in \eqref{model2} is constant with respect to its argument.
\item $X$ is ergodic and irreducible in $S$ with the stationary distribution $\pi:=\{\pi_{j}:j\in S\}.$
\end{enumerate}
\eas

Assumption \ref{as1}(a) is referred to as the hidden Markovian environment assumption. It follows from Assumptions \ref{as1}(b) that $X$ is positive recurrent. 
Fix a state $j\in S.$ Let $\tau^{j}_{0}$ be the first time instant $X$ hits state $j$ and stays there for $T^{j}_{0}$ time. One can recursively define, for $k\ge 1$,
\begin{align}
\tau^{j}_{k}:=\inf\big\{t>\tau^{j}_{k-1}+T^{j}_{k-1}: X_{t}=j\big\},\quad 
T^{j}_{k}:=\inf\big\{t>\tau^{j}_{k}:X_{t}\neq j\big\}-\tau^{j}_{k}, \quad
I^{j}_{k}:= (\tau^{j}_{k-1},\tau^{j}_{k}]. \label{tau}
\end{align}
$T^j_k$ is the time $X$ spends in state $j$ after hitting state $j$ at time $\tau^{j}_{k}$. $(T^{j}_{k})_{k\ge 0}$ is an i.i.d. sequence with $\exponential(-\lambda_{jj})$-distributed terms.
The renewal cycle lengths form an i.i.d. sequence $(|I^{j}_{k}|)_{k\ge 1}$.
A consequence of positive recurrence of $X$ is that $E|I^{j}_{k}|<\infty$ for any $(j,k)\in (S,\N)$. 
Let 
\begin{align}
g_{t}^{j}:=\max\big(\sup\{n\in\mathbb{N}:\tau^{j}_{n}\le t\},0\big), \quad \sup\emptyset:=-\infty, \label{g_{t}}
\end{align}
i.e. the number of times the chain $X$ revisits the state $j$ before time $t$. Positive recurrence of $X$ implies that $g_{t}^{j}\stackrel{\text{a.s.}}{\to}\infty$ as $t\to\infty$.

\section{The stable regime}\label{sec_stable}  

In this section we study long time behavior of the joint process $(Y,X):=(Y_t,X_t)_{t\ge 0}$ and processes defined in terms of certain functionals of $(X_t)_{t\ge 0}$ under conditions ensuring that convergence in distribution holds as $t\to\infty$. 
Together with Assumption \ref{as1}, the following assumption ensures the existence of a stationary distribution for $(Y,X)$:  
\bas\label{as2} The $S$-valued process $X$ and the functions $a,b:S\to\R$ satisfy 
\begin{enumerate}[(a)]
\item $a$ is integrable with respect to $\pi,$ and $E_{\pi} a(\cdot)>0.$
\item For every $j\in S$, $E\Big[\log^{+}\int_{\tau^{j}_{0}}^{\tau^{j}_{1}}b^{2}(X_{s})e^{-2\int_{s}^{\tau^{j}_{1}}a(X_{r})dr}ds\Big]<\infty.$
\end{enumerate}
\eas

\begin{remark}\label{rem1} 
Assumption \ref{as2} correspond, in the current setting, to the general condition \eqref{sre_conv_conds} for existence of a stationary solution to the stochastic recurrence equation
$$
Z_{j,n+1}=e^{-\int_{\tau^{j}_{n}}^{\tau^{j}_{n+1}}a(X_{s})ds}Z_{j,n}+\int_{\tau^{j}_{n}}^{\tau^{j}_{n+1}}b^{2}(X_{s})e^{-2\int_{s}^{\tau^{j}_{n+1}}a(X_{r})dr}ds
$$ 
with affine invariant solution of the form 
$$
Z_{j}\stackrel{d}{=}e^{-\int_{\tau^{j}_{0}}^{\tau^{j}_{1}}a(X_{s})ds}Z_{j}+\int_{\tau^{j}_{0}}^{\tau^{j}_{1}}b^{2}(X_{s})e^{-2\int_{s}^{\tau^{j}_{1}}a(X_{r})dr}ds.
$$ 
If Assumption \ref{as2}(a) holds but not Assumption \ref{as2}(b), then results similar to Theorem 1.1 of \cite{Iksa2015div-contractive} hold. 
Notice that if $\sup_{j\in S} |a(j)|<\infty$ and $E_{\pi}b^{2}(\cdot)<\infty$, then Assumption \ref{as2}(b) follows immediately from  Assumption \ref{as1}(b) as a consequence of the inequalities $\log^{+}|ab|\le \log^{+}|a|+\log^{+}|b|$ and $\log^{+}|a|\le |a|$ for any $a,b$.
\end{remark}

An explicit expression for the stationary distribution of the joint process $(Y,X)$ is the following.

\begin{Theorem}\label{P1}Under Assumptions \ref{as1} and \ref{as2} the stationary distribution of the joint process $(Y,X)$ can be expressed as a scale mixture of Gaussians of the following form
\beqn
(Y_{t},X_{t}) \stackrel{d}{\to} \Big(\sum_{j\in S}\delta_U(\{j\})Z_j,U\Big)\quad \text{as } t\to\infty \label{P1e1}
\eeqn
where $U\indep (Z_j)_{j\in S}$, $U\sim \pi$, $Z_j\stackrel{d}{=}\sqrt{V_j}N$, $V_j\indep N$, $N\sim\normal(0,1)$ and 
$$
V_j\stackrel{d}{=}
b^2(j)\int_0^{T^j}e^{-2a(j)(T^{j}-s)}ds+e^{-2a(j)T^{j}}V^{*}_j,
$$
where $T^{j}\sim \exponential(-\lambda_{jj})$ is independent of $V^{*}_j$, and 
$\mathcal{L}(V^{*}_j)$ is the unique solution to \eqref{x=ax+b} with $(A,B)$ having the distribution of 
\begin{align*}
\Big(e^{-2\int_{\tau^{j}_{0}}^{\tau_{1}^{j}}a(X_{s})ds},
\int_{\tau^{j}_{0}}^{\tau_{1}^{j}}b^{2}(X_{s})e^{-2\int_{s}^{\tau^{j}_{1}}a(X_{r})dr}ds\Big).
\end{align*}
\end{Theorem}

\begin{remark}
Theorem \ref{P1} may be generalized by instead of the model \eqref{model1} considering the more general Ornstein-Uhlenbeck model
\beqn
dY_{t}= \big(c(X_{t})- a(X_{t})Y_{t}\big)dt + b(X_{t})dW_{t}, \s Y_{0}=y\in{\R},\label{model12}
\eeqn 
for an arbitrary function $c:S\to\R.$ The stationary distribution of $(Y,X)$ can be determined under the assumptions of Theorem \ref{P1} and the additional assumption (ensured by $E_{\pi}|c(\cdot)|<\infty$): 
$$
E_{} \log^{+}\int_{\tau_{0}^{j}}^{\tau_{1}^{j}}|c(X_{s})|e^{-2\int_{s}^{\tau_{1}^{j}}a(X_{r})dr}ds<\infty
\quad\text{for all } j\in S.
$$ 
The stationary distribution is given by  
\begin{align*}
(Y_{t},X_{t}) \stackrel{d}{\to} \Big(\sum_{j\in S}\delta_U(\{j\})Z_j,U\Big)\quad \text{as } t\to\infty, 
\end{align*}
where $U\indep (Z_j)_{j\in S}$, $U\sim \pi$, $Z_j\stackrel{d}{=}M_j+\sqrt{V_j}N$, $(M_j,V_j)\indep N$, 
$N\sim\normal(0,1)$ and 
\begin{align*}
\begin{pmatrix} M_{j} \\ V_{j} \end{pmatrix}
\stackrel{d}{=} 
\begin{pmatrix} 
c(j)\int_{0}^{T^{j}}e^{-a(j)(T^{j}-s)}ds+e^{-a(j)T^{j}}M^{*}_{j} \\ 
b^{2}(j)\int_{0}^{T^{j}}e^{-2a(j)(T^{j}-s)}ds+e^{-2a(j)T^{j}}V^{*}_{j} 
\end{pmatrix},
\end{align*}
where $T^{j}\sim\exponential(-\lambda_{jj})$ is independent of $\big(M^{*}_j,V^{*}_j\big)$ and $\mathcal{L}\big(M^{*}_j,V^{*}_j\big)$ is the unique solution to 
\beqn 
\begin{pmatrix} M^{*}_{j} \\ V^{*}_{j} 
\end{pmatrix}
\stackrel{d}{=} 
\left[{\begin{array}{cc}
\sqrt{A_{j}} & 0 \\
0 &  A_{j}\\
\end{array}}\right]
\begin{pmatrix} M^{*}_{j} \\ V^{*}_{j} \end{pmatrix}
+\begin{pmatrix} C_{j} \\ B_{j} \end{pmatrix},
\quad
(M^{*}_j,V^{*}_j) \indep (A_j,B_j,C_j),
\label{P1e2}
\eeqn
with $(A_j,B_j,C_j)$ having the distribution of
\begin{align*}
\bigg(e^{-\int_{\tau^{j}_{0}}^{\tau_{1}^{j}}2a(X_{s})ds},\int_{\tau^{j}_{0}}^{\tau_{1}^{j}}b^{2}(X_{s})e^{-2\int_{s}^{\tau_{1}^{j}}a(X_{r})dr}ds,\int_{\tau^{j}_{0}}^{\tau_{1}^{j}}c(X_{s})e^{-\int_{s}^{\tau_{1}^{j}}a(X_{r})dr}ds\bigg)
\end{align*}
with notations used in Theorem \ref{P1}. This generalization of Theorem \ref{P1} follows from from expressing $Y$ in \eqref{model12} as
\begin{align}
Y_{t}=Y_{0}e^{-\int_{0}^{t}a(X_{r})dr}+\int_{0}^{t}c(X_{s})\big(e^{-\int_{s}^{t}a(X_{r})dr}\big)ds
+\int_{0}^{t}b(X_{s})e^{-\int_{s}^{t}a(X_{r})dr}dW_s. \label{ito2}
\end{align}
The characterization \eqref{P1e2} follows along the lines of the proof of Theorem \ref{P1} by modifying the proof of Lemma \ref{lem1} by determining the weak limit of 
$$
\bigg(e^{-\int_{0}^{t}a(X_{r})dr},\int_{0}^{t}c(X_{s})e^{-\int_{s}^{t}a(X_{r})dr}ds,\int_{0}^{t}b^{2}(X_{s})e^{-2\int_{s}^{t}a(X_{r})dr}ds\bigg).
$$ 
\end{remark}

\begin{remark}\label{Levy} 
Theorem \ref{P1} can be extended further by replacing the standard Brownian motion $W$
by an arbitrary L\'evy process $L$ in \eqref{model1}:
\beqn
dY_{t}= - a(X_{t})Y_{t}dt + b(X_{t})dL_{t}, \s Y_{0}=y\in\R. \label{levy}
\eeqn
Exact long time behavior can be determined from the expression
$$
Y_{t}=Y_{0}e^{-\int_{0}^{t}a(X_{r})dr}+\int_{0}^{t}b(X_{s})e^{-\int_{s}^{t}a(X_{r})dr}dL_{s}
$$ 
under the same assumptions as in Theorem \ref{P1} except that Assumption \ref{as2}(b) is replaced by 
$$
E \log^{+}\Big|\int_{\tau^{j}_{0}}^{\tau_{1}^{j}}b(X_{s})e^{-\int_{s}^{\tau_{1}^{j}}a(X_{r})dr}dL_{s}\Big|<\infty
\quad \text{for all } j\in S.
$$
The stationary distribution can be expressed as
$$
(Y_{t},X_{t}) \stackrel{d}{\to} \Big(\sum_{j\in S}\delta_U(\{j\})Z_j,U\Big)\quad \text{as } t\to\infty
$$
where $U\indep (Z_j)_{j\in S}$ and 
$$
Z_j\stackrel{d}{=}b(j)\int_{0}^{T^{j}}e^{-a(j)(T^{j}-s)}dL_{s}+e^{-a(j)T^{j}}Z^{*}_j,
$$ 
where
$T^{j}\sim \exponential(-\lambda_{jj})$, $L$ and $Z^{*}_j$ are independent, and 
$\mathcal{L}(Z^{*}_{j})$ is the unique solution to \eqref{x=ax+b} with $(A,B)$ having the distribution of 
\beqn
\bigg(e^{-\int_{\tau^{j}_{0}}^{\tau_{1}^{j}}a(X_{s})ds},\int_{\tau^{j}_{0}}^{\tau_{1}^{j}}b(X_{s})e^{-\int_{s}^{\tau_{1}^{j}}a(X_{r})dr}dL_s\bigg).\label{P1e22}
\eeqn
The special case $L=W$ corresponds to $Z_j\stackrel{d}{=}\sqrt{V_j}N$ with $(V_j,N)$ as in Theorem \ref{P1}. 
\end{remark}

\begin{remark}
Theorem \ref{P1} together with Mill's ratio inequalities yield tail bounds for the stationary distribution $\mathcal{L}(Y_{\infty})$, writing $Y_t\stackrel{d}{\to}Y_{\infty}$ as $t\to\infty$ for the marginal convergence in \eqref{P1e1}. With $\mu_j:=\mathcal{L}(V_j)$, 
$$
\sum_{j\in S}\pi_{j}\int_{\R^{+}}\frac{t\phi(t/\sqrt{\sigma})}{1+t^{2}}\mu_{j}(d\sigma) \le P[Y_{\infty}>t]
\le \sum_{j\in S}\pi_{j}\int_{\R^{+}}\frac{\phi(t/\sqrt{\sigma})}{t}\mu_{j}(d\sigma).
$$
Sharper versions of the Mill's ratio inequalities, see e.g. \cite{fan2012new}, yield sharper bounds.
\end{remark}

\begin{remark}
Moments for the stationary distribution of $Y$ in \eqref{levy} can be computed recursively using the representation for $(A_j,B_j)$ in \eqref{P1e22}. From $Z_j\stackrel{d}{=}A_jZ_j+B_j$ follows that, for $m\in\mathbb{N}$,  
$$
Z_j^m(1-A_j^m)=\sum_{k=0}^{m-1}{m \choose k}A_j^kB_j^{m-k}Z_j^k.
$$
If there exists $n\in\mathbb{N}$ such that $E|Z_j|^{n}<\infty$, then $EA_j^{m}<\infty$ and $EA_j^kB_j^{m-k}<\infty$ for $0\le k\le m\le n$, and independence between $Z_j$ and $(A_j,B_j)$ gives 
$$
EZ_j^{m}=\frac{1}{1-EA^{m}_{j}}\sum_{k=0}^{m-1}{m\choose k}E\big[A^{k}_{j}B^{m-k}_{j}\big]EZ^{k}_{j}.
$$
From the representation for the limit distribution follows that $EY^m=\sum_{j\in S}\pi_jEZ_j^m$.
\end{remark}

\begin{remark}\label{rem2} 
Theorem \ref{P1} can be generalized by allowing $Y$ to be a vector valued Ornstein Uhlenbeck process. 
In that case, when both drift and diffusion functions $a(\cdot), b(\cdot)$ are matrix valued functions of hidden Markov process $X$, stability conditions will change in a nontrivial way which require careful analysis. 

Theorem \ref{P1} and the methodology used for proving the theorem can be extended to general regime switching dynamics that is marginal of a Markov renewal process (also known as semi-Markov process), where in every regime $j\in S$ the regime process spends a random time distributed as $H^{j}$ with finite mean that is not Exponentially distributed.  
Since a semi-Markov process is in general non-Markovian, instead of a stationary distribution one should use the similar notion of a limiting distribution for investigating exact long time behavior. Allowing $H^{j}$ to have infinite mean would make the analysis substantially more complicated.
\end{remark}

In many applications integrals of the form 
\beqn
F_t:=\int_{0}^{t}d(X_{s})e^{-\int_{s}^{t}c(X_{r})dr}ds\label{It_exp_func}
\eeqn
appear for functions $c,d: S\to \R$ and $X$ being a regime process satisfying Assumption \ref{as1}. The following corollary addresses the long time behavior for $F_t$ under the stability regime $E_{\pi}c(\cdot)>0$ and suitable integral property of $d(\cdot)$ in form of the following assumptions.

\bas\label{as22} The $S$-valued process $X$ and the functions $c,d:S\to\R$ satisfy 
\begin{enumerate}[(a)]
\item $c$ is integrable with respect to $\pi,$ and $E_{\pi}c(\cdot)>0.$
\item For every $j\in S$, $E\Big[\log^{+}\Big|\int_{\tau^{j}_{0}}^{\tau_{1}^{j}}d^{}\big(X_{s}\big)e^{-\int_{s}^{\tau_{1}^{j}}c(X_{r})dr}ds\Big|\Big]<\infty.$
\end{enumerate}
\eas

The difference between Assumption \ref{as2} and Assumption \ref{as22} is that in the latter the function $d$ can take negative values in contrast to only positive values for $b^{2}$ appearing in Assumption \ref{as2}. 

\begin{Corollary}\label{Cor1}
Under Assumptions \ref{as1} and \ref{as22}, $F_t$ in \eqref{It_exp_func} satisfies 
\begin{align*}
F_t \,\stackrel{d}{\to} \, \sum_{j\in S}\delta_U(\{j\})V_j \quad \text{as } t\to\infty,
\end{align*}
where $U\indep (V_j)_{j\in S}$ and 
$$
V_j\stackrel{d}{=}d(j)\int_0^{T^{j}}e^{-c(j)(T^{j}-s)}ds+e^{-c(j)T_{j}}V^{*}_j,
$$
where $T^{j}\sim \exponential(-\lambda_{jj})$ is independent of $V^{*}_j$, and 
$\mathcal{L}(V^{*}_j)$ is the unique solution to \eqref{x=ax+b} with $(A,B)$ having the distribution of 
\begin{align}
\bigg(e^{-\int_{\tau^{j}_{0}}^{\tau_{1}^{j}}c(X_{s})ds},
\int_{\tau^{j}_{0}}^{\tau_{1}^{j}}d^{}(X_{s})e^{-\int_{s}^{\tau_{1}^{j}}c(X_{r})dr}ds\bigg).\label{MP1e21cd}
\end{align}
\end{Corollary}

\begin{remark}
The Goldie-Kesten theorem (Theorem 2.4.4 in \cite{buraczewski2016stochastic}) characterizes heavy-tailed behavior of the solution to the fixed-point equation \eqref{x=ax+b}. If $A\geq 0$ a.s. and $\mathcal{L}(\log A \mid A>0)$ is non-arithmetic, $P[Ax+B=x]<1$ for all $x\in\R$, and there exists $\nu>0$ such that 
$$
E A^{\nu}=1,\quad E|B|^{\nu}<\infty, \quad E A^{\nu}\log^+ A<\infty, 
$$
then there exists constants $c_+,c_-$ with $c_++c_->0$ such that 
$$
P[X>x]\sim c_+x^{-\nu}, \quad P[X<-x]\sim c_-x^{-\nu}\quad \text{as } x\to\infty.
$$
This result is applicable to the stationary distribution of $Y$ in Theorem \ref{P1} if $\inf_{j\in S}a(j)<0$. 
Let 
$$
\nu_{j}:=\sup\Big\{c>0: Ee^{-c\int_{\tau_{0}^{j}}^{\tau_{1}^{j}}a(X_{s})ds}<1\Big\}, \quad \nu^{*}:=\inf_{j\in S}\nu_j.
$$
If 
$$
\sup_{j\in S}\Big\{E\Big[\Big(\int_{\tau^{j}_{0}}^{\tau_{1}^{j}}b^{2}(X_{s})e^{-2\int_{s}^{\tau^{j}_{1}}a(X_{r})dr}ds\Big)^{\frac{\nu^{*}}{2}}\Big]\Big\}<\infty
$$
and
$$
\sup_{j\in S}\Big\{E\Big[e^{-\nu^{*}\int_{\tau_{0}^{j}}^{\tau_{1}^{j}}a(X_{s})ds}\log^+ e^{-2\int_{\tau_{0}^{j}}^{\tau_{1}^{j}}a(X_{s})ds}\Big]\Big\}<\infty,
$$
then the left and right tails of the symmetric distribution of 
$$
\sum_{j\in S}\delta_{U}(\{j\})\sqrt{V_j}N
$$ 
are regularly varying with index $\nu^{*}$. The statement follows since $V_j$ is a stochastic affine transformation of $V^{*}_j$ by random variables having finite moments of all orders, and since the standard normal distribution has finite moments of all orders. 
The indices $\nu_j$ can be estimated from the sample paths of $X$ through the empirical estimator 
$$
\widehat{\nu}_{n,j}:=\inf\Big\{c>0:\frac{1}{n}\sum_{i=1}^{n}e^{-c\int_{I^{j}_{i}}a(X_{s})ds}=1\Big\}, 
\quad \inf\emptyset:=+\infty, 
$$ 
of $\nu_{j}$ based on $n$ regenerating intervals $\{I_{i}^{j}\}_{i=1}^{n}.$ 
Therefore, $\nu^{*}$ may be estimated iteratively as the limit of  
$$
\widehat{\nu}^{*}_{n,{j_{k+1}}}:=\min\Big(\widehat{\nu}^{*}_{n,{j_{k}}},\inf\Big\{c\in \big(0,\widehat{\nu}^{*}_{n,{j_{k}}}\big):\frac{1}{n}\sum_{i=1}^{n}e^{-c\int_{I^{j}_{i}}a(X_{s})ds}=1\Big\}\Big), \quad k=1,\dots,|S|.
$$
Another representation for the tail index $\inf_{j\in S}\nu_j$ was presented in \cite{bardet2010long}, \cite{saporta2005tail} (with a spectral analysis) for finite state space $S$, in terms of the spectral radius of a certain matrix. 
\end{remark}

\section{Transient and null-recurrent regimes}\label{sec_unstable}

In this section we study the long time behavior of the process $Y=(Y_t)_{t\ge 0}$ and processes defined in terms of certain functionals of $(X_t)_{t\ge 0}$ under conditions different from Assumption \ref{as2} and Assumption \ref{as22}. In particular, it will be assumed that the stability condition $E_{\pi}a(\cdot)> 0$ in Assumption \ref{as2}(a) does not hold and that instead $E_{\pi}a(\cdot)\le 0$. By choosing a suitable Lyapunov function as done in \cite{jinghai2015criteria} it follows that $E_{\pi}a(\cdot)<0$ and $E_{\pi}a(\cdot)=0$ correspond to transience and null-recurrence, respectively, for the model \eqref{model1}.  

It can be shown that if Assumption \ref{as1} holds and the stability condition $E_{\pi}a(\cdot)> 0$ in Assumption \ref{as2}(a) is replaced by $E_{\pi}a(\cdot)<0$, then the long-term behavior of $Y$ will be determined by the first term in the representation 
\begin{align*}
Y_{t}=Y_{0}e^{-\int_{0}^{t}a(X_{r})dr}+\int_{0}^{t}b(X_{s})e^{-\int_{s}^{t}a(X_{r})dr}dW_{s}. 
\end{align*}
Consequently, the ergodic theorem gives 
$$
\frac{\log|Y_{t}|}{t}\stackrel{\text{a.s.}}{\to}-E_{\pi}a(\cdot)\quad \text{as } t\to\infty.
$$
However it is not well known how scaled fluctuation 
$$
\frac{\log|Y_{t}|}{\sqrt{t}}+\sqrt{t}E_{\pi}a(\cdot)
$$ 
behave for the model \eqref{model1} as $t\to \infty$, and how the regime switching dynamics play a role in that limit. This is the motivation behind the results of the present section.

\bas\label{as3}
The $S$-valued process $X$ and the functions $a,b:S\to\R$ satisfy 
\begin{enumerate}[(a)]
\item $a$ is integrable with respect to $\pi$, and $E_{\pi} a(\cdot)\le 0$.
\item For every $j\in S$, $\sigma_j^2:=\Var\Big(\int_{\tau_{0}^{j}}^{\tau_{1}^{j}}a(X_{s})ds\Big)<\infty$.
\item Assumption \ref{as2}(b) holds.
\end{enumerate}
\eas

\begin{remark} \label{R1T2}
Assumption \ref{as3}(b) is necessary for having the CLT type results in Theorem \ref{T2} below. 
For $|S|=2$ or for very simple cyclic Markov chains, e.g. $S=\{0,\dots,n\}$ and transitions $0\to1\to \dots\to n\to 0$, Assumption \ref{as3}(b) holds trivially but for more general cases we need to impose Assumption \ref{as3}(b). 
\end{remark}

\bt\label{T2}  
Suppose Assumptions \ref{as1} and \ref{as3} hold. 
Let $U\sim \pi$ and $N\sim \normal(0,1)$ be independent. 
\begin{enumerate}[(a)]
\item (\textbf{Transient regime}) 
If $E_{\pi}a(\cdot)<0$, then
\begin{align*}
\frac{\log |Y_{t}|}{\sqrt{t}}+\sqrt{t}E_{\pi}a(\cdot) 
\stackrel{d}{\to}
\sum_{j\in S} \delta_U(\{j\})\frac{\sigma_{j}}{\sqrt{E|I^{j}_1|}}N \quad \textrm{as } t\to\infty. 
\end{align*}
\item (\textbf{Null recurrent regime}) 
If $E_{\pi}a(\cdot)=0$, $b(\cdot)\neq 0$ and  
$$
E\Big[\Big(\log \int_{\tau_{0}^{j}}^{\tau_{1}^{j}}b^{2}\big(X_{s}\big)e^{-2\int_{s}^{\tau_{1}^{j}}a(X_{r})dr}ds\Big)^{2}\Big]<\infty,
$$ 
then 
\begin{align}
\frac{\log |Y_{t}|}{\sqrt{t}} \stackrel{d}{\to}
\sum_{j\in S} \delta_U(\{j\})\frac{\sigma_{j}}{\sqrt{E|I^{j}_1|}}|N| \quad \textrm{as } t\to\infty. \label{nullrec1}
\end{align}
\end{enumerate}
\et
In Theorem \ref{T2}(b), $|N|$ appears in the weak limit because the left-hand side in \eqref{nullrec1} asymptotically behaves as a scaled mixture of maxima of partial sums of random walks for which the long time behavior was characterized by Erd\"os and Kac in \cite{ErdosKac1946}.   

\begin{remark}
The result in Theorem \ref{T2}(a) does not depend on the diffusion function $b$. In fact, the result holds for any stochastic process $(Y_t)_{t\ge 0}$ such that 
$$
Y_t=\widetilde{Y}_t e^{-\int_{0}^{t}a(X_{r})dr}
\quad\text{where}\quad 
\widetilde{Y}_t^{\frac{1}{\sqrt{t}}}\stackrel{d}{\to}\delta_{1} \quad\text{as } t\to\infty.
$$
An example is 
$$
\widetilde{Y}_t=Y_{0}+\int_{0}^{t}b(X_{s})e^{\int_{0}^{s}a(X_{r})dr}dL_{s},
$$
for a L\'evy process $L$ together with an associated integrability condition. 
\end{remark}

For the exponential integral process $(F_t)_{t\ge 0}$ in \eqref{It_exp_func} results similar to Theorem \ref{T2} hold. Similar results are found in Theorem 2(a) and Theorem 3(a) of \cite{hitczenko2011renorming}.

\bas\label{as5}
The $S$-valued process $X$ and the functions $c,d:S\to\R$ satisfy 
\begin{enumerate}[(a)]
\item $c$ is integrable with respect to $\pi$, and $E_{\pi} c(\cdot)\le 0$.
\item For every $j\in S$, $\sigma_j^2:=\Var\Big(\int_{\tau_{0}^{j}}^{\tau_{1}^{j}}c(X_{s})ds\Big)<\infty$.
\item Assumption \ref{as22}(b) holds.
\end{enumerate}
\eas

\begin{Corollary}\label{Cor3}
Suppose that Assumptions \ref{as1} and \ref{as5} hold. 
Let $U\sim \pi$ and $N\sim \normal(0,1)$ be independent. 
\begin{enumerate}[(a)]
\item (\textbf{Transient regime}) If $E_{\pi}c(\cdot)<0$, then
\begin{align*}
\frac{\log |F_t|}{\sqrt{t}}+\sqrt{t}E_{\pi}\big(c(\cdot)\big) 
\stackrel{d}{\to} \sum_{j\in S}\delta_{U}(\{j\})\frac{\sigma_{j}}{\sqrt{E|I^{j}_1|}}N
\quad\text{as }t\to\infty. 
\end{align*}
\item (\textbf{Null recurrent regime}) If $E_{\pi}c(\cdot)=0$, $d\neq 0$ and 
$$
E\Big[\Big(\log^{+} \int_{\tau_{0}^{j}}^{\tau_{1}^{j}}d^{}\big(X_{s}\big)e^{-\int_{s}^{\tau_{k}^{j}}c(X_{r})dr}ds\Big)^{2}\Big]<\infty,
$$
then 
\begin{align*}
\frac{\log |F_t|}{\sqrt{t}}
\stackrel{d}{\to} \sum_{j\in S}\delta_{U}(\{j\})\frac{\sigma_{j}}{\sqrt{E|I^{j}_1|}}|N|
\quad\text{as } t\to\infty. 
\end{align*}
\end{enumerate}
\end{Corollary}
Corollary \ref{Cor3} can be proved by a minor modification of the proof of Theorem \ref{T2}. 
The proof of Corollary \ref{Cor3} is therefore omitted.

\section{Applications}\label{sec_applications} 

The results presented in Sections \ref{sec_stable} and \ref{sec_unstable} have applications in various contexts involving stochastic processes under hidden Markovian environments. We consider two specific applications: the Cox-Ingersoll-Ross process originally introduced for modeling interest rates, and a classical SIS model considered in epidemiology.

\subsection{The Cox-Ingersoll-Ross process}

In \cite{hou2017heavy} authors considered general Cox-Ingersoll-Ross (CIR) model and explored the criteria for different tail properties of the stationary distribution in terms of the hidden Markovian switching  dynamics. Following we consider a specific parametrization of CIR process and express the explicit stationary distribution that one observes under hidden Markovian contexts for drift and diffusion coefficients. 

Let $a,b:S\to\R$ with $a\neq 0$ be arbitrary functions, let $n\in\mathbb{N}$ such that $n\ge 2$ and define
$$
\kappa,\theta,\xi:S\to\R, \quad
\kappa:=2a,\quad \theta:= \frac{nb^{2}}{2a}, \quad \xi:=2b. 
$$
Consider the CIR process $(R_t)_{t\ge 0}$ defined as the solution to the stochastic differential equation
\beqn
dR_{t}= \kappa(X_{t})(\theta(X_{t}) - R_{t})dt+\xi(X_{t}) \sqrt{R_{t}} dW_{t}, \s R_{0}=r_0>0.\label{cir}
\eeqn
In \cite{CIR1985} the CIR model was introduced as a interest rate model where $\kappa,\theta,\xi$ are positive constants 
and Feller proved that $2\kappa\theta\ge \xi^{2}$ ensures that the CIR process is non-negative with probability $1$ (which holds automatically for the above parametrization, regardless of the value of $X_{t}=x\in S$ since $n\ge 2$). We present long time results for the CIR process defined above. 

\begin{Proposition}\label{CIRapp}
Consider the CIR model \eqref{cir} with above parametrizations and $X$ as defined in \eqref{model2}. 
Let $U\sim\pi$ and $N_1,\dots,N_n\sim \normal(0,1)$ be independent.
\begin{enumerate}[(a)]
\item Suppose Assumptions \ref{as1} and \ref{as2} ($E_{\pi}\kappa(\cdot)>0$ and integrability of $\xi(\cdot)$) hold. Then  
\beqn
R_{t} \stackrel{d}{\to} \sum_{j\in S}\delta_{U}(\{j\})V_{j}\sum_{i=1}^n N_i^2
\quad\text{as }t\to\infty,
\label{cir1}
\eeqn
where $V_{j}$ is independent of $U,N_1,\dots,N_n$ and $\mathcal{L}(V_j)$ is given in Theorem \ref{P1}. 
\item Suppose Assumptions \ref{as1} and \ref{as3} hold. If $E_{\pi}\kappa(\cdot)<0$, then
$$
\frac{\log R_{t}}{\sqrt{t}}+2\sqrt{t}E_{\pi}\kappa(\cdot) \stackrel{d}{\to} \sum_{j\in S}\delta_{U}(\{j\})\frac{2\sigma_{j}}{\sqrt{E|I^{j}_1|}}\max_{1\le i\le n}N_{i}
\quad\text{as }t\to\infty.
$$
\item Suppose Assumptions \ref{as1} and \ref{as3} hold. If $E_{\pi}\kappa(\cdot)=0$, then
$$
\frac{\log R_{t}}{\sqrt{t}} \stackrel{d}{\to} \sum_{j\in S}\delta_{U}\big(\{j\}\big)\frac{2\sigma_{j}}{\sqrt{E|I^{j}_1|}}\max_{1\le i\le n}|N_{i}| 
\quad\text{as }t\to\infty.
$$
\end{enumerate}
\end{Proposition}

\begin{remark}\label{R1CIR}
Proposition \ref{CIRapp} cannot be easily extended to noninteger values $n$. However for noninteger $n$ the solution to the CIR process can be written as a sum of squared $OU$ processes and a squared Bessel process (Chapter 6 of \cite{jeanblanc2009mathematical}). Therefore we would need to investigate long time behavior of Bessel processes under Markovian regime switching in detail in order to generalize Proposition \ref{CIRapp}.
\end{remark}

\subsection{The SIS model in epidemiology}

We consider deterministic SIS epidemic models under Markov modulated environments similar to the one considered in \cite{gray2012sis} but with number of regimes $|S|>2$ instead of just two states. Let $\alpha,\beta:S\to\R$ be functions denoting rate of infection and recovery, respectively. Consider a fixed population size $n$ and subpopulation sizes $I_{t}$ and $S_{t}$ at time $t\ge 0$, satisfying $I_{t}+S_{t}=n$, of infectious and susceptible individuals, respectively. 
The model is determined by the system of equations
\begin{align*}
\frac{dS_t}{dt}=-\beta(X_t)S_tI_t+\alpha(X_t)I_t, \quad
\frac{dI_t}{dt}=\beta(X_t)S_tI_t - \alpha(X_t)I_t, \quad I_t+S_t=n,
\end{align*}
where $I_t,S_t$ may take arbitrary real values in $[0,n]$. 
If $\alpha,\beta$ were constants instead of functions then one would have, with $\gamma:=\beta n-\alpha$,  
$$
I_t=\Big[e^{-\gamma t}\Big(\frac{1}{I_0}-\frac{\beta}{\gamma}\Big)+\frac{\beta}{\gamma}\Big]1_{\{\gamma\neq 0\}}
+\Big[\frac{1}{I_0}+\beta t\Big]^{-1}1_{\{\gamma=0\}}
$$
which would lead to 
$$
\lim_{t\to\infty}I_t=\frac{\gamma}{\beta} \quad\text{if } \gamma>0, \quad
\lim_{t\to\infty}I_t=0 \quad\text{if } \gamma\le 0.
$$
Notice that $\gamma\le 0$ is equivalent to $R_{0}\le 1$, where $R_{0}$ is the reproduction number. For $|S|=2$, results from \cite{gray2012sis} suggest 
\begin{enumerate}[(a)]
\item If $E_{\pi}\gamma(\cdot)>0,$ there is a persistence for the infected population size, i.e., in probability, 
$$
\liminf_{t\to\infty}\log I_t\le \frac{E_{\pi}\gamma(\cdot)}{E_{\pi}\beta(\cdot)},\quad 
\limsup_{t\to\infty}\log I_t\ge \frac{E_{\pi}\gamma(\cdot)}{E_{\pi}\beta(\cdot)}.
$$
\item If $E_{\pi}\gamma(\cdot)=0,$ simulation suggests that $I_t\to 0$ in probability, but no analytical results were provided. 
\item If $E_{\pi}\gamma(\cdot)<0,$ then $I_t\to 0$ exponentially fast, i.e $\limsup_{t\to\infty}\frac{\log I_t}{t}\le E_{\pi}\gamma(\cdot)$, in probability.
\end{enumerate}

The following proposition provides sharper asymptotic results.

\begin{Proposition}\label{SIS} 
For the model described above and for $a<b$ the following holds
\begin{enumerate}[(a)]
\item If $E_{\pi}\gamma(\cdot)>0$, then $\lim_{t\to\infty}P\Big[\frac{1}{I_{t}}\in (a,b)\Big]=\sum_{j\in S}\pi_{j}P\Big[V_{j}\in (a,b)\Big]$, where $V_{j}$ is given in Corollary \ref{Cor1} when $(c,d)=(\gamma,\beta)$.
\item If $E_{\pi}\gamma(\cdot)=0$, then, for $a\ge 0$, $\lim_{t\to\infty}P\Big[I_{t}\in \big(e^{-b\sqrt{t}},e^{-a\sqrt{t}}\big)\Big]=\sum_{j\in S}\pi_{j} P\Big[|N|\in (a\sigma_{j},b\sigma_{j})\Big]$, where $N\sim \normal(0,1)$ and $\sigma_{j}^2=\Var\big(\int_{\tau^{j}_0}^{\tau^{j}_{1}}\gamma(X_s)ds\big)$.
\item If $E_{\pi}\gamma(\cdot)<0$, then $\lim_{t\to\infty}P\Big[I_{t}\in \big(e^{tE_{\pi}\gamma(\cdot)-b\sqrt{t}},e^{tE_{\pi}\gamma(\cdot)-a\sqrt{t}}\big)\Big]=\sum_{j\in S}\pi_{j} P\big[N\in (a\sigma_{j},b\sigma_{j})\big]$, where $N\sim \normal(0,1)$ and $\sigma_{j}^2=\Var\big(\int_{\tau^{j}_0}^{\tau^{j}_{1}}\gamma(X_s)ds\big)$.
\end{enumerate}
\end{Proposition}

Similar results can be obtained for Markov modulated deterministic SIR models as integrals of type $\int_{0}^{t}e^{-\int_{s}^{t}\gamma(X_{s})}\beta(X_{s})ds$ show up as a consequence of the Markovian environment in the transition rates.

\section{Proofs}\label{sec_proofs}

We use the convention $\sum_{i=j}^k a_i=0$ and $\prod_{i=j}^k a_i=1$ if $j>k$ for any $a_i$.

For functions $c,d:S\to\R$ and $j\in S$, define 
\beqn
G^{c,d}_{j}(x):=\int_{0}^{x}d(j)e^{-c(j)(x-s)}ds=x d(j)1_{\{c(j)=0\}}+ \frac{d(j)}{c(j)}\Big(1-e^{-xc(j)}\Big)1_{\{c(j)\neq 0\}}.
\label{Gfun}
\eeqn

\subsection{Proof of Theorem \ref{P1}}

\begin{proof}
We prove the statement in a number of steps. 
The marginal distribution of the regime process $X$ in for $S$ with initial distribution $\delta_{i}$, $i\in S$, is denoted 
$$
P_{ij}(0,t):=P\big[X_{t}=j\mid X_{0}=i\big]=\delta_{i}P_{t}(i,\{j\}).
$$
Let the transition kernels of the time-homogeneous Markov process $(X_{t},Y_{t})_{t\ge 0}$ be the maps $\mb{P^{t}}:\big(S\times\R\big) \times \mathcal{B}\big(S\times\R\big)\to [0,1]$. Then, for any 
$(i,y_{0})\in (S\times \R)$ and $(I,A)\in \mathcal{B}(S\times \R)$, 
\begin{align}
\delta_{(i,y_0)}\mb{P^{t}}\big((i,y_{0}),(I,A)\big)
&=\sum_{j\in I}P\big[(X_{t},Y_{t})\in (\{j\},A)\mid (X_0,Y_0)=(i,y_{0})\big]\non\\
&=\sum_{j\in I}P_{ij}\big(0,t\big)P\big[Y_t\in A\mid Y_0=y_0,X_{0}=i,X_{t}=j \big], \label{t1e1}
\end{align}
since, by \eqref{causal}, 
$$
P\big[X_{t}=j\mid (X_{0},Y_{0})=(i,y_0)\big]=P\big[X_{t}=j\mid X_{0}=i\big]=P_{ij}\big(0,t\big).
$$
If we can show that there exists a $\mu_{\infty}\in\mathcal{P}(S\times R)$ such that 
\beqn
 \delta_{(i,y_{0})}\mb{P^{t}}\mathop{\to}^{w} \mu_{\infty}\quad\text{as } t\to\infty\quad \text{for all } (i,y_{0})\in (S\times \R),\label{toprove}
\eeqn
then as a consequence of the strong Feller property, satisfied trivially by $(X,Y)$, one can deduce that $\mu_{\infty}$ is the unique invariant measure. The strong Feller property ensures that one may interchange the order of limit and expectation in 
$$
\mu_{\infty}\mb{P^{s}}f=\lim_{t\to\infty}\delta_{(i,y_{0})}\mb{P^{t}}\mb{P^{s}}f=\lim_{t\to\infty}\delta_{(i,y_{0})}\mb{P^{t+s}}f=\mu_{\infty}f \quad\text{for all } s>0. 
$$
We find $\mu_{\infty}$ using \eqref{toprove}.

Lemma \ref{lem1} below provides a representation of the second factor in the product in \eqref{t1e1}. The representation is expressed in terms of an $(\mathcal{F}_{t}^X)_{t\geq 0}$-adapted stochastic process $(Q^{(1)}_{t},Q^{(2)}_{t})_{t\geq 0}$. 
Fix two arbitrary states $i,j \in S$ and suppose that $X_{0}=i.$ Recall that $\tau_{0}^{j}$ is the first time $X$ hits state $j$ and define $\tau_{k}^{j}$ and $T_{k}^{j}$ for $k\ge 0$ recursively as in \eqref{tau}. For $k\ge 1$ define
\begin{align}
(J^{j}_{k},K^{j}_{k})&\stackrel{}{:=}\Big(e^{-\int_{\tau_{k-1}^{j}}^{\tau_{k}^{j}}a(X_{s})ds},\int_{\tau_{k-1}^{j}}^{\tau_{k}^{j}}b^{2}(X_{s})e^{-2\int_{s}^{\tau_{k}^{j}}a(X_{r})dr}ds\Big), \label{JK}\\ 
(J^{(i,j)}_{0},K^{(i,j)}_{0})&\stackrel{}{:=}\Big(e^{-\int_{0}^{\tau_{0}^{j}}a(X_{s})ds},\int_{0}^{\tau_{0}^{j}}b^{2}(X_{s})e^{-2\int_{s}^{\tau_{0}^{j}}a(X_{r})dr}ds\Big).\non
\end{align}
Note that $((J^{j}_{k},K^{j}_{k}))_{k=1}^{\infty}$ is an i.i.d sequence by virtue of a renewal argument for the regenerating process $X$: the renewal cycles $(I^{j}_{k})_{k\ge 1}$ are i.i.d. $X$ generates identical pairs $(J_{k}^{j},K_{k}^{j})$ of functionals of $X$ which are independent for $k\ge 1$. However, for a fixed $t>0,$ $((J^{j}_{k},K^{j}_{k}))_{k=1}^{g^{j}_{t}}$ are not i.i.d. since $g^{j}_{t}$, defined in \eqref{g_{t}}, is a renewal time which is dependent on sum of all renewal cycle lengths before time $t$.

\begin{lemma}\label{lem1} 
Suppose Assumptions \ref{as1} and \ref{as2} hold. Then for any $t>0,$ 
\begin{align*}
P\big[Y_{t} \in A \mid Y_{0}=y_{0},X_{0}=i,X_{t}=j\big]=P\big[y_{0}Q^{(1)}_{t}+\sqrt{Q^{(2)}_{t}}N\in A \mid X_0=i,X_t=j\big],
\end{align*}
where $N\sim \normal(0,1)$, $N\indep (\mathcal{F}_{t}^{X})_{t\ge 0}$ and $(Q^{(1)}_t,Q^{(2)}_t)_{t\ge 0}$ is an $(\mathcal{F}_{t}^{X})_{t\ge 0}$-adapted process given by 
\begin{align}
 Q^{(1)}_t&=e^{ -a(j)(t-\tau_{g_{t}^{j}}^{j})}\bigg(\prod_{k=1}^{g^{j}_{t}}J^{j}_{k}\bigg)J^{(i,j)}_{0}\,,\non\\
 Q^{(2)}_t&=G_{j}\big(t-\tau_{g_{t}^{j}}^{j}\big)+e^{-2a(j)(t-\tau_{g_{t}^{j}}^{j})}\bigg[\bigg(\prod_{k=1}^{g_{t}^{j}-1}J^{j}_{k}\bigg)^{2} K^{(i,j)}_{0}+\sum_{k=1}^{g_{t}^{j}-1} \bigg( \prod_{l=k+1}^{g_{t}^{j}-1}J^{j}_{l}\bigg)^{2} K^{j}_{k}\bigg]. \label{joint}
\end{align}
\end{lemma}

Now we prove the main result by finding $\mu_{\infty}$ by computing the limit as $t\to \infty$ for a term in the sum \eqref{t1e1}. Note that $P_{ij}(0,t)\to \pi_{j}$ as $t\to\infty$ by the ergodicity Assumption \ref{as1}(b). 
Using Lemma \ref{lem1} it remains to show that
\beqn
P\Big[y_{0}Q^{(1)}_{t}+\sqrt{Q^{(2)}_{t}}N\in A \mid X_{0}=i,X_{t}=j\Big]\to P\Big[\sqrt{V_j}N\in A\Big]\s
\quad (A,j)\in \big(\mathcal{B}(\R),S\big). \label{toprove2}
\eeqn
Let us summarize the steps. \eqref{toprove2} follows if we show 
\beqn
\mathcal{L}\big(\big(Q^{(1)}_{t},Q^{(2)}_{t}\big)\mid X_{0}=i,X_{t}=j\big)\stackrel{w}{\to}\delta_{0}\otimes \mathcal{L}\big(V_{j}\big)
\quad\text{as } t\to\infty.
\label{toprove4}
\eeqn 
The result above is observed by taking the limit as $t\to \infty$ of the expression on the right hand side of \eqref{joint}. 
Set
\beqn
R_t:=\big(R_{t}^{(1)}, R_{t}^{(2)}\big):=\bigg(
\bigg(\prod_{k=1}^{g_{t}^{j}}J_{k}^{j}\bigg)J^{(i,j)}_{0},
\bigg(\prod_{k=1}^{g_{t}^{j}}J_{k}^{j}\bigg)^{2}K^{(i,j)}_{0}+\sum_{k=1}^{g_{t}^{j}}\bigg(\prod_{l=k+1}^{g_{t}^{j}}J^{j}_{l}\bigg)^{2}K^{j}_{k}\label{condR_t}
\bigg)
\eeqn
and notice that, on $\{X_{0}=i,X_{t}=j\}$,  
\beqn
\big(Q^{(1)}_{t},Q^{(2)}_{t}\big)
=\bigg(e^{-a(j)\big(t-\tau^{j}_{g^{j}_{t}}\big)}R^{(1)}_{t},G^{2a,b^2}_{j}(t-\tau_{g_{t}^{j}}^{j})+e^{-2a(j)(t-\tau_{g_{t}^{j}}^{j})}R^{(2)}_{t}\bigg).\label{z_to_r}
\eeqn
Next we determine the weak limit, as $t\to\infty$, of  
$
\mathcal{L}\big(\big((t-\tau^{j}_{g_{t}^{j}}),R_{t}\big)\mid X_{t}=j\big)
$
and in particular show that $(t-\tau^{j}_{g_{t}^{j}})$ and $R_t$ are asymptotically independent, given $X_{t}=j$.
Take $x\in (0,t)$ and set $B_{t,x}:=\{X\text{ makes no jump in } (t-x,t]\}$. Then, 
$$
\{X_{t}=j, t-\tau^{j}_{g_{t}^{j}}>x,R_{t}\in A\}=\{B_{t,x}, X_{t-x}=j, R_{t-x}\in A\}, \quad A:=(A_{1},A_{2}),
$$ 
since $\tau^{j}_{g_{t}^{j}}<t-x$ implies $\tau^{j}_{g_{t-x}^{j}}=\tau^{j}_{g_{t}^{j}}$ which further implies $R_{t}=R_{t-x}.$ 
Using the above equality, 
\begin{align*}
P\big[t-\tau^{j}_{g_{t}^{j}} > x, R_{t}\in A\mid X_{t}=j\big]
&= P\big[B_{t,x}, X_{t-x}=j, R_{t-x}\in A\mid X_{t}=j\big] \\
&= P\big[B_{t,x}\mid X_{t-x}=j, R_{t-x} \in A\big]P\big[R_{t-x}\in A\mid X_{t-x}=j\big] \frac{P[X_{t-x}=j]}{P[X_{t}=j]}
\end{align*}
Since $R_{t-x}$ is $\mathcal{F}^X_{t-x}$-measurable, the Markov property of $X$ implies 
$$
P\big[B_{t,x}\mid X_{t-x}=j, R_{t-x} \in A\big]=P\big[B_{t,x}\mid X_{t-x}=j\big]=e^{\lambda_{jj}x}.
$$ 
Moreover, 
$$
\lim_{t\to\infty}\frac{P[X_{t-x}=j]}{P[X_{t}=j]}=1.
$$
We will show that 
\beqn
\lim_{t\to\infty}P\big[R_{t-x}\in A\mid X_{t-x}=j\big]=P\big[X^{*}_{j}\in A_{2}\big]1_{\{0\in A_{1}\}},
\quad A:=(A_{1},A_{2}), \label{condition1}
\eeqn
for any $A:=(A_{1},A_{2})$ such that $P[(0,X^{*}_j)\in \partial A]=0$ with $\partial A$ denoting the boundary of $A$, 
where $X^{*}_{j}$ satisfies 
\begin{align}
 \sum_{k=1}^{g_{t}^{j}} \bigg( \prod_{l=k+1}^{g_{t}^{j}}J^{j}_{l}\bigg)^{2}K^{j}_{k}  \stackrel{d}{\to} X^{*}_{j}
 \quad \text{as } t\to\infty,  
 \quad 
 X^{*}_{j}\stackrel{d}{=}(J_{1}^{j})^{2}X^{*}_{j}+K_{1}^{j},
 \quad 
 X^{*}_{j} \indep \big(J_{1}^{j},K_{1}^{j}\big).\label{toprove3Xj}
\end{align}
This would imply that $(t-\tau^{j}_{g_{t}^{j}})$ and $R_{t}$, given $\{X_{t}=j\}$,  are asymptotically independently, and the weak limit of $(t-\tau^{j}_{g_{t}^{j}})$ is $\exponential(-\lambda_{jj})$.  
We will use the following lemma:

\begin{lemma}\label{lem2} Under Assumptions \ref{as1} and \ref{as2}, 
$R_t$ in \eqref{condR_t} satisfies
\beqn
R_{t}\stackrel{d}{\to}\big(0,X^{*}_{j}\big)\quad\text{as } t\to\infty,\label{lem2eeq1}
\eeqn
where $X^{*}_{j}\stackrel{d}{=}\big(J_{1}^{j}\big)^{2}X^{*}_{j}+K_{1}^{j}$ with $X^{*}_{j} \indep\big(J_{1}^{j},K_{1}^{j}\big).$ \end{lemma}

Exploiting the Markovian structure of $X$ we will show that 
\beqn 
\lim_{t\to\infty}P\big[R_{t-x}\in A\mid X_{t-x}=j\big]=P\big[ (0,X_{j}^{*})\in A\big].\label{condition2}
\eeqn  
We begin with the following lemma for a slightly more general process $R$ in order to fit well also in subsequent proofs.
In Section \ref{verification2.01} we show that the $R$ in the current context defined in \eqref{condR_t} indeed satisfies conditions \eqref{condition2.01} of the following lemma.

\begin{lemma}\label{lem1.5} 
Consider an $(\mathcal{F}^X_{t})_{t\ge 0}$-adapted process $(R_{t})_{t\ge 0}$, and suppose there exist a random variable $R_{\infty}$ and an increasing deterministic function $t\mapsto \varepsilon(t)$ such that 
\beqn
R_{t}\stackrel{d}{\to}R_{\infty},\quad \varepsilon(t)\to\infty, \quad \frac{\varepsilon(t)}{t}\to 0 \quad \text{and}\quad R_{t} - R_{t- \varepsilon(t)} \stackrel{P}{\to} 0 \quad \text{as } t\to\infty.
\label{condition2.01}
\eeqn
Suppose further that $\big(L^{(1)}_{t},L^{(2)}_{t}\big)_{t\ge 0}$ satisfies $\big(L^{(1)}_{t},L^{(2)}_{t}\big)\stackrel{P}{\to} (0,1)$ as $t\to\infty$.
Then, for any $A$ such that $P[R_{\infty}\in \partial A]=0$, 
$$
\lim_{t\to\infty}P\big[L^{(1)}_{t}+L^{(2)}_{t}R_{t}\in A\mid X_{t}=j\big]=P\big[R_{\infty}\in A\big].
$$
\end{lemma}

\begin{remark}
Note that for establishing weak convergence it is sufficient to restrict attention to continuity sets, i.e. here sets satisfying $P[R_{\infty}\in \partial A]=0$. 
The part of \eqref{condition2.01} involving $\varepsilon(t)$ resembles Anscombe's condition (\cite{gut2009stopped}, p.~16). 
As seen in the subsequent arguments using Lemma \ref{lem1.5}, the $R_t$ to which Lemma \ref{lem1.5} is applied will be 
$\mathcal{F}^X_{\tau_{g_{t}^{j}}}$ measureable. 
\end{remark}

Using Lemma \ref{lem2} we have shown \eqref{condition1} and in Subsection \ref{verification2.01} it is verified that the conditions 
\eqref{condition2.01} are fulfilled. 
Using Lemma \ref{lem2} in \eqref{z_to_r} together with the aforementioned argument prove that 
$$
\mathcal{L}\big(\big(Q^{(1)}_{t},Q^{(2)}_{t}\big)\mid X_{0}=i, X_{t}=j\big) \stackrel{w}{\to} \delta_{0}\otimes \mathcal{L}\Big(G^{2a,b^2}_{j}\big(T^{j}\big)+e^{-2a(j)T^{j}}X^{*}_{j}\Big),
$$ 
where $T^{j}\sim\exponential(-\lambda_{jj})$ is independent of $X^{*}_{j}$. 
This proves \eqref{toprove4} and concludes the proof of Theorem \ref{P1}.

\subsection{Proof of Lemma \ref{lem1}}

\begin{proof}
The process $(Y_{t})_{t\geq 0}$ in \eqref{model1} has the representation 
\beqn
Y_{t}=\mb{\Phi(0,t)}Y_{0}+\int_{0}^{t}b(X_{s})\mb{\Phi(s,t)}dW_{s}, \quad \mb{\Phi(s,t)}:=e^{-\int_{s}^{t}a(X_{r})dr}.\label{ito1}
\eeqn
Define 
\beqn
(Q^{(1)}_{t},Q^{(2)}_{t}):=\bigg(\mb{\Phi(0,t)},\int_{0}^{t}b^{2}(X_{s})\mb{\Phi^{2}(s,t)}ds\bigg).\label{z_1z_2}
\eeqn
Note that on $\{X_{0}=i,X_{t}=j\}:=\{\omega\in\Omega : X_{0}(\omega)=i,X_{t}(\omega)=j\}$,
\beqn
Q^{(1)}_{t}=\mb{\Phi(0,t)}
=\mb{\Phi(\tau^{j}_{g_{t}^{j}},t)}\bigg(\prod_{l=1}^{g_{t}^{j}}\mb{\Phi(\tau_{l-1}^{j},\tau_{l}^{j})}\bigg)\mb{\Phi(0,\tau_{0}^{j})}
=e^{-a(j)\big(t-\tau^{j}_{g^{_j}_{t}}\big)}\bigg(\prod_{l=1}^{g_{t}^{j}} J^{j}_{l}\bigg)J_{0}^{(i,j)},\label{J-t-1}
\eeqn
where the identity 
$\mb{\Phi(\tau^{j}_{g_{t}^{j}},t)}=e^{-a(j)\big(t-\tau^{j}_{g^{_j}_{t}}\big)}$ on $\{X_{0}=i,X_{t}=j\}$  
follows from the fact that if $s\in (\tau^{j}_{g_{t}^{j}},t)$, then $X_s=j.$

We now consider the second term in \eqref{ito1}. Notice that by Ito isometry 
$$
\mathcal{L}\Big(\int_{0}^{t}b(X_{s})\mb{\Phi(s,t)}dW_{s}\mid X_0=i,X_t=j\Big)
=\mathcal{L}\Big(\Big[\int_{0}^{t}b^2(X_{s})\mb{\Phi}^2(s,t)ds\Big]^{1/2}N\mid X_0=i,X_t=j\Big)
$$
Partitioning $[0,t]$ into 
$[0,t]=[0,\tau_{0}^{j}]\cup \mathop{\cup}_{k=1}^{g_{t}^{j}}(\tau_{k-1}^{j},\tau_{k}^{j}]\cup (\tau_{g_{t}^{j}}^{j},t],$  
we may write $Q^{(2)}_t$ as 
\beqn
 \int_{0}^{t}b^{2}(X_{s})\mb{\Phi^{2}(s,t)}ds=\int_{0}^{\tau_{0}^{j}}b^{2}(X_{s})\mb{\Phi^{2}(s,t)}ds+\sum_{k=1}^{g_{t}^{j}}\int_{\tau_{k-1}^{j}}^{\tau_{k}^{j}}b^{2}(X_{s})\mb{\Phi^{2}(s,t)}ds+\int_{\tau_{g_{t}^{j}}^{j}}^{t}b^{2}(X_{s})\mb{\Phi^{2}(s,t)}ds.\label{mte1}
\eeqn
Expanding one term in the middle sum
\begin{align}
\int_{\tau_{k-1}^{j}}^{\tau_{k}^{j}}b^{2}(X_{s})\mb{\Phi^{2}(s,t)}ds
&=\int_{\tau_{k-1}^{j}}^{\tau_{k}^{j}}\mb{\Phi^{2}(\tau_{k}^{j},t)}b^{2}(X_{s})\mb{\Phi^{2}(s,\tau_{k}^{j})}ds\non\\
&=\mb{\Phi^{2}(\tau_{g_{t}^{j}}^{j},t)}\Big(\prod_{i=k}^{g_{t}^{j}-1} \mb{\Phi^{2}(\tau_{i}^{j},\tau_{i+1}^{j})}\Big)\int_{\tau_{k-1}^{j}}^{\tau_{k}^{j}}b^{2}(X_{s})\mb{\Phi^{2}(s,\tau_{k}^{j})}ds\non
\end{align}
for some $k\le g_{t}^{j}.$ Putting these estimates in \eqref{mte1} we have 
\begin{align*}
\int_{0}^{t}b^{2}(X_{s})\mb{\Phi^{2}(s,t)}ds 
=&\mb{\Phi^{2}_{}(\tau_{g_{t}^{j}}^{j},t)}\bigg[ \Big(\prod_{l=0}^{g_{t}^{j}-1} \mb{\Phi^{2}_{}(\tau_{l}^{j},\tau_{l+1}^{j})}\Big)\int_{0}^{\tau_{0}^{j}}b^{2}(X_{s})\mb{\Phi^{2}(s,\tau_{0}^{j})}ds \\
&\quad\quad\quad\quad\quad+\sum_{k=1}^{g^{j}_{t}}\Big(\prod_{l=k}^{g_{t}^{j}-1} \mb{\Phi}^{2}_{}(\tau_{l}^{j},\tau_{l+1}^{j})\Big)\int_{\tau_{k}^{j}-1}^{\tau_{k}^{j}}b^2(X_s)\mb{\Phi^{2}}_{}(s,\tau_{k}^{j})ds\bigg] \\
&+\int_{\tau_{g_{t}^{j}}^{j}}^{t}b^{2}(X_{s})\mb{\Phi^{2}(s,t)}ds.
\end{align*}
On $\{X_{0}=i,X_{t}=j\}$, 
\begin{align*}
\int_{\tau_{g_{t}^{j}}^{j}}^{t}b^{2}(X_{s})\mb{\Phi^{2}(s,t)}ds
=b^{2}(j)\int_{\tau_{g_{t}^{j}}^{j}}^{t}e^{-2(t-s)a(j)}ds 
=G^{2a,b^2}_{j}(t-\tau_{g_{t}^{j}}^{j}).
\end{align*}
Moreover, on $\{X_{0}=i,X_{t}=j\}$, 
$$
\int_{0}^{\tau_{g_{t}^{j}}^{j}}b^{2}(X_{s})\mb{\Phi^{2}(s,t)}ds
=e^{-2a(j)(t-\tau_{g_{t}^{j}}^{j})}\bigg[\bigg(\prod_{k=1}^{g_{t}^{j}}J^{j}_{k}\bigg)^{2} K^{(i,j)}_{0}
+\sum_{k=1}^{g_{t}^{j}}\bigg(\prod_{l=k+1}^{g_{t}^{j}}J^{j}_{l}\bigg)^{2} K^{j}_{k}\bigg].
$$
Combining the above gives, on $\{X_{0}=i,X_{t}=j\}$,   
\beqn
Q^{(2)}_{t}=G^{2a,b^2}_{j}(t-\tau_{g_{t}^{j}}^{j})
+e^{-2a(j)(t-\tau_{g_{t}^{j}}^{j})}\bigg[\bigg(\prod_{k=1}^{g_{t}^{j}}J_{k}^{j}\bigg)^{2}K^{(i,j)}_{0}
+\sum_{k=1}^{g_{t}^{j}} \bigg(\prod_{l=k+1}^{g_{t}^{j}}J^{j}_{l}\bigg)^{2} K^{j}_{k}\bigg].\label{J-t-2}
\eeqn
The assertion of Lemma \ref{lem1} follows by combining \eqref{J-t-1} and \eqref{J-t-2}. 
\hfill $\square$
\end{proof}

\subsection{Proof of Lemma \ref{lem2}}
In following two subsections we prove Lemma \ref{lem2} by first proving \eqref{toprove3Xj} and then
\begin{align}
 \Big(\prod_{k=1}^{g_{t}^{j}}J_{k}^{j}\Big)J^{(i,j)}_{0} \stackrel{d}{\to} 0, 
 \quad 
 \Big(\prod_{k=1}^{g_{t}^{j}}J_{k}^{j}\Big)^{2} K^{(i,j)}_{0} \stackrel{d}{\to} 0 
 \quad \text{as } t\to\infty.
 \label{toprove3}
\end{align}
 
\subsubsection{Proof of \eqref{toprove3Xj}}

To simplify the notation we omit the superscript $j$ in this subsection. Main technical difficulty here is to obtain a similar version of Vervaat's Theorem 1.5 in \cite{vervaat1979stochastic} but in a continuous time context driven by the renewal time $g_{t}$.  For $n\ge 1$, define 
$$
\widetilde{S}_{n}:= \sum_{k=1}^{n}\bigg(\prod_{i=k+1}^{n}J^{}_{i}\bigg)^{2}K_{k}
=K^{}_{n}+J^{2}_{n}K^{}_{n-1}+\ldots+(J^{}_{2}\ldots J^{}_{n})^{2}K^{}_{1}.
$$ 
Note that $\widetilde{S}_{n+1}=\widetilde{S}_nJ_{n+1}^2+K_{n+1}$ and therefore $(\widetilde{S}_n)_{n\ge 1}$ is Markovian. 
For $n\ge 1$, define
$$
S_{n}:=\sum_{k=1}^{n}\bigg(\prod_{i=1}^{k-1}J_{i}\bigg)^{2}K_{k}=K_{1}+J_{1}^{2}K_{2}+\ldots+(J_{1}\ldots J_{n-1})^{2}K_{n}.
$$ 
Note that $S_{n+1}=S_{n}+\big(J^{2}_{1}\ldots J^{2}_{n}\big) K_{n+1}$ and therefore $(S_{n})_{n\ge 1}$ is not Markovian but $S_n$ is a partial sum of the infinite sum 
$$
S_{\infty}:= \sum_{k=1}^{\infty}\bigg(\prod_{i=1}^{k-1}J_{i}\bigg)^{2} K_{k}.
$$ 
We will show that $\widetilde{S}_{g_{t}}\stackrel{d}{\to}S_{\infty}$ as $t\to \infty$ where $S_{\infty}$ satisfies 
$S_{\infty}\stackrel{d}{=}(J_{1}^{j})^{2}S_{\infty}+K_{1}^{j}$, 
$S_{\infty} \indep \big(J_{1}^{j},K_{1}^{j}\big)$. From uniqueness of the solution we will conclude that 
$\mathcal{L}(S_{\infty})=\mathcal{L}(X^{*}_{j})$. 
We will prove \eqref{toprove3Xj} in the following steps.

\begin{itemize}
\item Step $1:$ $\widetilde{S}_{g_{t}}\stackrel{d}{=}S_{g_{t}}$ for any $t>0$. 
\item Step $2:$ $\frac{\log K_{g_{t+1}}}{g_{t}}\stackrel{\text{a.s.}}{\to} 0$ as $t\to \infty$. 
\item Step $3:$ $S_{\infty}\stackrel{d}{=}\sum_{k=1}^{\infty}\Big(\prod_{i=1}^{k-1}J_{i+g_{t}+1}\Big)^{2}K_{k+g_{t}+1}$.
\item Step $4:$ $\mathcal{L}(S_{\infty})$ is the unique solution of \eqref{toprove3Xj}. 
\end{itemize}
Given steps 1-4, proof of \eqref{toprove3Xj} is completed as follows. Observe by step 1 $\mathcal{L}(S_{g_{t}})=\mathcal{L}(\tilde{S}_{g_{t}});$ and
\begin{align}
S_{\infty}-S_{g_{t}}
=J^{2}_{1}J^{2}_{2}\ldots J^{2}_{g_{t}}K_{g_{t}+1} +J^{2}_{1}J^{2}_{2}\ldots J^{2}_{g_{t}+1}\Big(\sum_{k=1}^{\infty}\Big(\prod_{i=1}^{k-1}J_{i+g_{t}+1}\Big)^{2}K_{k+g_{t}+1}\Big).\label{beta}
\end{align}
We will prove that the quantity in RHS of \eqref{beta} converges to zero in probability as $t\to\infty$. Then the assertion $\tilde{S}_{g_{t}}\stackrel{d}{\to}S_{\infty}$ follows. The first term in the quantity in the expectation may be written as 
$$
J^{2}_{1}J^{2}_{2}\ldots J^{2}_{g_{t}}K_{g_{t}+1}
=\exp\Big\{t\Big(\frac{2\sum_{i=1}^{g_{t}}\log J_{i}}{t}+\frac{\log K_{g_{t}+1}}{t}\Big)\Big\}
=\exp\Big\{t\Big(\frac{2\sum_{i=1}^{g_{t}}\log J_{i}}{g_{t}}+\frac{\log K_{g_{t}+1}}{g_t}\Big)\frac{g_{t}}{t}\Big\}.
$$ 
By the renewal theorem, 
$$
\frac{g_{t}}{t}\stackrel{\text{a.s.}}{\to}\frac{1}{E|I_{1}|}\quad\text{as }t\to\infty
$$ 
and by Assumption \ref{as2}(a), as $t\to\infty$,  
$$
\frac{\sum_{i=1}^{g_{t}}\log J_{i}}{g_{t}}\stackrel{\text{a.s.}}{\to} E\log J_{i} 
= -E\int_{\tau_{0}^{}}^{\tau_{1}^{}}a(X_{s})ds=- E|I_{1}|E_{\pi}a(\cdot)<0.
$$ 
Together with Step 2, we have therefore
$$
\Big(\frac{2\sum_{i=1}^{g_{t}}\log J_{i}}{g_{t}}+\frac{\log K_{g_{t}+1}}{g_t}\Big)\frac{g_{t}}{t}
\stackrel{\text{a.s.}}{\to} -2E_{\pi}a(\cdot)<0\quad \text{as } t\to\infty
$$
from which $J^{2}_{1}J^{2}_{2}\ldots J^{2}_{g_{t}}K_{g_{t}+1}\stackrel{\text{a.s.}}{\to} 0$ as $t\to\infty$ follows.
For the second term in the quantity in the expectation, 
$$
\mathcal{L}\Big(\sum_{k=1}^{\infty}\Big(\prod_{i=1}^{k-1}J_{i+g_{t}+1}\Big)^{2} K_{k+g_{t}+1}\Big)=\mathcal{L}(X^{*}_j)
$$
and 
$$
J^{2}_{1}J^{2}_{2}\ldots J^{2}_{g_{t}}=\exp\Big\{t\Big(\frac{2\sum_{i=1}^{g_{t}}\log J_{i}}{g_{t}}\Big)\frac{g_{t}}{t}\Big\}
\stackrel{\text{a.s.}}{\to} 0\quad \text{as } t\to\infty
$$
from which Slutsky's theorem gives
$$
J^{2}_{1}J^{2}_{2}\ldots J^{2}_{g_{t}+1}\Big(\sum_{k=1}^{\infty}\Big(\prod_{i=1}^{k-1}J_{i+g_{t}+1}\Big)^{2}K_{k+g_{t}+1}\Big)
\to 0\quad \text{in probability as } t\to\infty.
$$
Given Steps 1-4 we have thus shown \eqref{toprove3Xj}. It only remains to prove Steps 1-4. 

\begin{itemize}
\item Step $1:$  
Notice that for each $t>0$, 
$$
\big((J_{1},K_{1}),(J_{2},K_{2}),\ldots,(J_{g_{t}},K_{g_{t}})\big)\stackrel{d}{=}\big((J_{g_{t}},K_{g_{t}}),(J_{g_{t}-1},K_{g_{t}-1}),\ldots,(J_{1},K_{1})\big)
$$
and that $\widetilde{S}_{\tau_{g_{t}}}$ and $S_{\tau_{g_{t}}}$ are the result of applying the same function to the two above identically distributed random vectors. Hence, $\widetilde{S}_{g_{t}}\stackrel{d}{=}S_{g_{t}}$ for any $t>0$. 
\item Step $2:$ We first show $\frac{\log K_{n+1}}{n}\stackrel{a.s}{\to}0$ as a consequence of Borel-Cantelli Lemma and $E\log K_{n}<\infty.$ This is evident as $E\log K^{j}_{n+1}<\infty$ implies
\begin{align*}
&\sum_{n=1}^{\infty}P\big[\log K^{j}_{n+1}>\varepsilon n\big]<\infty \quad\text{for every } \varepsilon>0,\\
&\quad \Leftrightarrow P\big[\log K^{j}_{n+1}>\varepsilon n \text{ i.o.}\big]=0\quad\text{for every } \varepsilon>0,\\
&\quad \Leftrightarrow \frac{1}{n}\log K^{j}_{n+1}\stackrel{\text{a.s.}}{\to} 0\quad\text{as }n\to\infty.
\end{align*}

Then from above step and $g_{t}^{j}\stackrel{\text{a.s.}}{\to}\infty$  Step 2 will follow by observing 
$$\Big\{\frac{\log K_{g_{t}^{j}+1}}{g_{t}^{j}}\not\to 0\Big\}= \Big\{\frac{\log K_{n+1}}{n}\not\to 0\Big\}\cup\{g_{t}^{j}\not\to\infty\}.$$

\item Step $3:$ 
Observe that for any $t>0$,  
\beqn
\big((J_{g_{t}+1},K_{g_{t}+1}),(J_{g_{t}+2},K_{g_{t}+2}),\ldots \big)\stackrel{d}{=}\big((J_{1},K_{1}),(J_{2},K_{2}),\ldots\big).\label{s2}
\eeqn
Notice that $\{g_{t}=n\}=\{\tau_{n}\le t<\tau_{n+1}\}$ is $\sigma\{X_{s}: 0<s <\tau_{n+1}\}$-measurable and independent of $\sigma\{X_{s}: \tau_{n+1}\le s\}$. 
For a sequence of measurable sets $(A_{i})_{i\ge 1}$,  
\begin{align}
&P\big((J_{g_{t}+1},K_{g_{t}+1})\in A_{1},(J_{g_{t}+2},K_{g_{t}+2})\in A_{2},\ldots\big)\non\\ 
&\quad=\sum_{n\ge 0}P\big((J_{g_{t}+1},K_{g_{t}+1})\in A_{1},(J_{g_{t}+2},K_{g_{t}+2})\in A_{2},\ldots\mid g_{t}=n\big)P(g_{t}=n)\non\\
&\quad=\sum_{n\ge 0}P\big((J_{n+1},K_{n+1})\in A_{1},(J_{n+2},K_{n+2})\in A_{2},\ldots\mid g_{t}=n\big)P(g_{t}=n)\non\\
&\quad=\sum_{n\ge 0}P\big((J_{n+1},K_{n+1})\in A_{1},(J_{n+2},K_{n+2})\in A_{2},\ldots \big)P(g_{t}=n)\s\s\text{(by independence)}\non\\
&\quad=\sum_{n\ge 0}P\big((J_{1},K_{1})\in A_{1},(J_{2},K_{2})\in A_{2},\ldots \big)P(g_{t}=n)\non\\
&\quad=P\big((J_{1},K_{1})\in A_{1},(J_{2},K_{2})\in A_{2},\ldots \big)\non
\end{align}
proving \eqref{s2}. A consequence of \eqref{s2} is 
\beqn
S_{\infty}\stackrel{d}{=} \sum_{k=1}^{\infty}\Big(\prod_{i=1}^{k-1}J_{i+g_{t}+1}\Big)^{2}K_{k+g_{t}+1}
=K_{g_{t}+1}+J^{2}_{g_{t}+1}\big(K_{g_{t}+2}+J^{2}_{g_{t}+2}\big(\ldots\big)\big).\label{s2'}
\eeqn
\item Step $4:$ We will prove that $\mathcal{L}(S_{\infty})$ is the unique solution to the distributional fixed point equation in \eqref{toprove3Xj}. 
By step $1,$ $\mathcal{L}\big(\widetilde{S}_{g_{t}}\big)=\mathcal{L}\big(S_{g_{t}}\big)$ holds. To prove the uniqueness of $\mathcal{L}(S_{\infty})$ as the solution to \eqref{toprove3Xj} under Assumption \ref{as1}, observe that 
$$
S_{\infty}:= \sum_{k=1}^{\infty}\Big(\prod_{i=1}^{k-1}J^{2}_{i}\Big) K_{k}
=K_{1}+J^{2}_{1}\Big(\sum_{k=1}^{\infty}\Big(\prod_{i=1}^{k-1}J^{2}_{i+1}\Big)^{} K_{k+1}\Big), 
$$
and 
$$
(K_{1},J_{1})\indep \Big(\sum_{k=1}^{\infty}\Big(\prod_{i=1}^{k-1}J^{2}_{i+1}\Big)^{} K_{k+1}\Big),
\quad \sum_{k=1}^{\infty}\Big(\prod_{i=1}^{k-1}J^{2}_{i+1}\Big)^{} K_{k+1}\stackrel{d}{=}S_{\infty}.
$$ 
So $S_{\infty}$ satisfies \eqref{toprove3Xj}. Now from Lemma 1.4(a) (which applies since $E\log J <0$ under Assumption \ref{as1}) and Vervaat's Theorem 1.5 in \cite{vervaat1979stochastic} (or Theorem 2.1 in \cite{goldie2000stability}) we conclude that the solution to \eqref{toprove3Xj} is unique and hence the assertion follows.
\end{itemize}

\begin{remark}We acknowledge that Lemma \ref{lem2} can be proved alternatively using Anscombe's theorem (page 16 of  \cite{gut2009stopped})  for any sequence of stopping times $g_{t}\to\infty$, and $R_{\tau_{n}}\stackrel{d}{\to}(0,X_{j}^{*})$. However, in order to conclude that $R_{t}\stackrel{d}{\to}(0,X_{j}^{*})$ one needs to verify Anscombe's conditions for $(R_{\tau_{n}})_{n\ge 1}$ leading to arguments similar to steps 1-4 above which are similar to the arguments in the proof of Varvaat's Theorem 1.5 in \cite{vervaat1979stochastic}.
\end{remark}

\subsubsection{Proof of \eqref{toprove3}}

We begin by showing that $\log^{}|J^{(i,j)}_{0}|=O_{P}(1)$ and $\log^{}|K^{(i,j)}_{0}|=O_{P}(1)$. Let, for arbitrary $i,j\in S$,  $A^{(k)}_{ij}$ denote the event that $X$ visits the state $j$ in $k$-th excursion from state $i$ to itself (i.e in time interval $(\tau_{k-1}^{i},\tau_{k}^{i}]$). Clearly $0<P\big(A^{(k)}_{ij}\big)<1$ for any $k\in\mathbb{N}$ since $X$ is irreducible in $S$. Note that on $\{X_{0}=i\}$, 
$$
\log^{}|J^{(i,j)}_{0}|=-\int_{0}^{\tau_{0}^{j}}a(X_{s})ds\le \sum_{l=1}^{k^{*}-1}\log J_{k}^{i}+ \int_{\tau_{k^{*}-1}^{i}}^{\tau_{k^{*}}^{i}}|a(X_{s})|ds,
\quad k^{*}=\inf\Big\{k\ge 1: 1_{\{A^{(k)}_{ij}\}}=1\Big\}.
$$ 
Since the upper bound is a geometric sum of random variables having finite expectation we have shown that $\log^{}|J^{(i,j)}_{0}|=O_{P}(1)$.
Similarly, on $\{X_{0}=i\}$,  
$$
K^{(i,j)}_{0}=\int_{0}^{\tau_{0}^{j}}b^{2}(X_{s})e^{-2\int_{s}^{\tau_{0}^{j}}a(X_{r})dr}ds=e^{(\tau_{0}^{j}- \tau_{k^{*}}^{i})}\sum_{l=1}^{k^{*}-1}\Big(\prod_{n=l+1}^{k^{*}}J_{n}^{i}\Big)^{2} K_{l}^{i}+\int_{\tau_{k*-1}^{i}}^{\tau_{0}^{j}}b^{2}(X_{s})e^{-2\int_{s}^{\tau_{0}^{j}}a(X_{r})dr}ds.
$$
Note that $(\tau_{k*-1}^{i},\tau_{0}^{j}]\subset(\tau_{k*-1}^{i},\tau_{k*}^{i}]$ and that 
$$
\sum_{l=1}^{\infty}\Big(\prod_{n=l+1}^{\infty}J_{n}^{i}\Big)^{2}
$$
is absolutely convergent due to Assumption \ref{as2}. Hence, $\log^{}|K^{(i,j)}_{0}|=O_{P}(1)$.

Note that 
$$
\bigg(\prod_{k=1}^{g_{t}^{j}-1}J_{k}^{j}\bigg)J^{(i,j)}_{0}=\exp\Big\{t\Big(\frac{\sum_{k=1}^{g^{j}_{t}-1}\log J^{j}_{k}}{g^{j}_{t}}+\frac{\log J^{(i,j)}_{0}}{g_t}\Big)\frac{g^{j}_{t}}{t}\Big\}.
$$ 
Applying the renewal theorem and law of large number in the renewal context gives, by Assumption \ref{as2}(a),  
$$
\frac{g_{t}}{t}\stackrel{\text{a.s.}}{\to}\frac{1}{E|I_{1}|},\quad\,\,\,\,\,\,\frac{\sum_{k=1}^{g^{j}_{t}-1}\log J^{j}_{k}}{g^{j}_{t}}\stackrel{\text{a.s.}}{\to} E\log J_{1}^{j} <0\quad \text{as } t\to\infty.
$$ 
Moreover,  
$\log J^{(i,j)}_{0}=O_{P}(1)$ implies 
$$
\frac{\log J^{(i,j)}_{0}}{g_t}\stackrel{\text{a.s.}}{\to} 0\quad \text{as } t\to\infty.
$$ 
Putting the pieces together yields the first statement of \eqref{toprove3}. 
Analogous arguments proves the second statement of \eqref{toprove3}:
$$
\Big(\prod_{k=1}^{g_{t}^{j}-1}J_{k}^{j}\Big)^{2} K^{(i,j)}_{0}
=\exp\Big\{t\Big(\frac{2\sum_{k=1}^{g^{j}_{t}-1}\log J^{j}_{k}}{g^{j}_{t}}+\frac{\log K^{(i,j)}_{0}}{g_t}\Big)\frac{g^{j}_{t}}{t}\Big\}\stackrel{\text{a.s.}}{\to}0\quad \text{as } t\to\infty.
$$
The proof of \eqref{toprove3} is complete and therefore also the proof Lemma \ref{lem2} is complete.
\hfill $\square$
\end{proof}

\subsection{Proof of Lemma \ref{lem1.5}}

\begin{proof}
The main assertion will follow if we prove that, for any $A$ whose boundary $\partial A$ satisfies $P\big[R_{\infty}\in \partial A\big]=0$,   
\begin{align}
\lim_{t\to\infty}P\big[R_{t}\in A, X_{t}=j\big] =\pi_{j}\lim_{t\to\infty}P\big[R_{t}\in A\big]=\pi_{j}P\big[R_{\infty}\in A\big].\label{cond0.02}
\end{align}
Then, by Slutsky's theorem, 
$$
\lim_{t\to\infty}P\big[L^{(1)}_{t}+L^{(2)}_{t}R_{t}\in A, X_{t}=j\big] =\lim_{t\to\infty}P\big[R_{t}\in A, X_{t}=j\big]. 
$$
We prove \eqref{cond0.02} in the following two steps.
\begin{itemize}
\item Step 1: $\lim_{t\to\infty}P\big[R_{t-\varepsilon(t)}\in A\mid X_{t}=j\big]=P[R_{\infty}\in A].$
\item Step 2: $\lim_{t\to\infty}P\big[R_{t}\in A\mid X_{t}=j\big]=\lim_{t\to\infty}P\big[R_{t-\varepsilon(t)}\in A\mid X_{t}=j\big].$ 
\end{itemize}
\textbf{Step 1}:
Notice that 
\begin{align*}
P\big[R_{t-\varepsilon(t)}\in A\mid X_{t}=j\big]
&=\sum_{j'\in S}P\big[R_{t-\varepsilon(t)}\in A\mid X_{t-\varepsilon(t)}=j', X_{t}=j\big]\frac{P\big[X_{t-\varepsilon(t)}=j', X_{t}=j\big]}{P[X_{t}=j]}\\
&= \sum_{j'\in S}P\big[R_{t-\varepsilon(t)}\in A\mid X_{t-\varepsilon(t)}=j'\big]\frac{P\big[X_{t}=j\mid X_{t-\varepsilon(t)}=j'\big]P[X_{t-\varepsilon(t)}=j']}{P[X_{t}=j]}\\
&= \sum_{j'\in S}P\big[R_{t-\varepsilon(t)}\in A, X_{t-\varepsilon(t)}=j'\big]\frac{P_{j'j}(0,\varepsilon(t))}{P[X_{t}=j]}. 
\end{align*}
The second equality follows by observing $R_{t-\varepsilon(t)}$ is conditionally independent of $X_t$ given $X_{t-\varepsilon(t)}$.
The third equality follows as $X$ is a time homogenous continuous time Markov chain. By ergodicity of $X$, 
$\lim_{t\to\infty}P_{j'j}(0,\varepsilon(t))/P[X_{t}=j]=1$. Since $R_{t-\varepsilon(t)}\stackrel{d}{\to}R_{\infty}$ as $t\to\infty$ the proof of Step 1 is complete. 

\textbf{Step 2}: Take any $\delta>0$ and define $A^{\delta}:=\{x: d(x,A)<\delta\},A^{-\delta}:=\{x\in A: d(x,A^c)<\delta\}$ and note that $A^{-\delta}\subseteq A\subseteq A^{\delta}$. Denoting $B_{t}:=R_{t}-R_{t-\varepsilon(t)}$ one has 
\begin{align}
P\big[R_{t}\in A \mid X_{t}=j\big]=P\big[R_{t-\varepsilon(t)}+B_{t}\in A, |B_{t}|\le \delta \mid X_{t}=j\big]+C_{t}\label{cond0.03}
\end{align}
where 
$$
\limsup_{t\to\infty}C_{t}\le \limsup_{t\to\infty}\frac{P[|B_{t}|>\delta]}{P[X_{t}=j]}=0.
$$  
The first term on the right-hand side in \eqref{cond0.03} satisfies 
$$
P\big[R_{t-\varepsilon(t)}\in A^{-\varepsilon}\mid X_{t}=j\big] \le P\big[R_{t-\varepsilon(t)}+B_{t}\in A, |B_{t}|\le \varepsilon \mid X_{t}=j\big]\le P\big[R_{t-\varepsilon(t)}\in A^{\varepsilon}\mid X_{t}=j\big].
$$ 
The assertion follows from Step 1 together with letting $\delta\to 0$.
\hfill $\square$

\end{proof}
\subsection{Verification of $R_{t}$ in \eqref{condR_t} satisfying \eqref{condition2.01} in Lemma \ref{lem1.5}}\label{verification2.01}
\begin{proof} Note that for $t>\tau_{1}^{j}$, on $\{X_{0}=i,X_{t}=j\}$ $R_{t}$ is composed of $\{(J_{k}^{j},K_{k}^{j}):k=1,\ldots,g_{t}; J_{0}^{(i,j)},K_{0}^{(i,j)}\}$. 
By Lemma \ref{lem2}, $R_{\infty}\stackrel{d}{=}(0,X^{*}_{j})$. We will show the condition $R_{t}- R_{t-\varepsilon(t)} \stackrel{P}{\to}0$ for an increasing function $\varepsilon$ such that $\lim_{t\to\infty}\varepsilon(t)=\infty$ and $\lim_{t\to\infty}\varepsilon(t)/t=0$. Note that from \eqref{condR_t}, 
$$
R_{t}- R_{t-\varepsilon(t)}=\big(R^{(1)}_{t}- R^{(1)}_{t-\varepsilon(t)},R^{(2)}_{t}- R^{(2)}_{t-\varepsilon(t)}\big)
$$ 
and the conclusion holds if we prove $R^{(i)}_{t}- R^{(i)}_{t-\varepsilon(t)}\stackrel{P}{\to} 0$ as $t\to\infty$ for $i=1,2$. Note that 
$$
R^{(1)}_{t}- R^{(1)}_{t-\varepsilon(t)}=\exp\bigg\{\sum_{k=1}^{g^j_{t-\varepsilon(t)}} \log J_{k}^{j}\bigg\}
\bigg[\exp\bigg\{\sum_{k=g^j_{t-\varepsilon(t)}+1}^{g^j_{t}} \log J_{k}^{j}\bigg\}-1\bigg]J_{0}^{(i,j)}.
$$ 
where by virtue of $E\log J_{k}^{j}<0.$ 
The exponents in the first two factors are partial sums of a random walk with negative drift, since $E\log J_{k}^{j}<0$. Hence, the first factor goes to $1$ and the second to $0$ almost surely by the law of large numbers in the renewal setting: 
$$
\frac{\sum_{k=1}^{g^j_{t-\varepsilon(t)}} \log J_{k}^{j}}{g^j_{t-\varepsilon(t)}}\stackrel{\text{a.s.}}{\to} E\log J_{k}^{j}<0,\quad \frac{\sum^{g^j_{t}}_{k=g^j_{t-\varepsilon(t)}+1} \log J_{k}^{j}}{g^j_{t}-g^j_{t-\varepsilon(t)}}\stackrel{\text{a.s.}}{\to} E\log J_{k}^{j}<0
\quad \text{as } t\to\infty
$$
since 
$$
\varepsilon(t)\to\infty, \quad t- \varepsilon(t)\to\infty, \quad \frac{g^j_{t}-g^j_{t-\varepsilon(t)}}{\varepsilon(t)}\stackrel{\text{a.s.}}{\to} \frac{1}{E|I^{j}|}
\quad \text{as } t\to\infty.
$$ 
Regarding $R^{(2)}_{t}- R^{(2)}_{t-\varepsilon(t)}$, applying arguments of Step 1 in Lemma \ref{lem2} one has 
$$
\begin{pmatrix} R^{(2)}_{t} \\  R^{(2)}_{t-\varepsilon(t)}\end{pmatrix}
=\begin{pmatrix} 
\big(\prod_{k=1}^{g_{t}^{j}}J_{k}^{j}\big)^{2}K^{(i,j)}_{0}+\tilde{S}_{g_{t}}\\ 
\big(\prod_{k=1}^{g_{t-\varepsilon(t)}^{j}}J_{k}^{j}\big)^{2}K^{(i,j)}_{0}+\tilde{S}_{g_{t-\varepsilon(t)}}
\end{pmatrix} 
\stackrel{d}{=}\begin{pmatrix}
\big(\prod_{k=1}^{g_{t}^{j}}J_{k}^{j}\big)^{2}K^{(i,j)}_{0}+S_{g^j_{t}}\\ \big(\prod_{k=1}^{g_{t-\varepsilon(t)}^{j}}J_{k}^{j}\big)^{2}K^{(i,j)}_{0}+S_{g^j_{t-\varepsilon(t)}}
\end{pmatrix}.
$$
Combining $\big(\prod_{k=1}^{g_{t}^{j}}J_{k}^{j}\big)^{2}K^{(i,j)}_{0}\stackrel{\text{a.s.}}{\to} 0$ and $S_{g^j_{t}}\stackrel{P}{\to} S_{\infty}$ as $t\to\infty$ from \eqref{beta} of Lemma \ref{lem2}, $S_{g^j_{t}}-S_{g^j_{t-\varepsilon(t)}}\stackrel{P}{\to} 0$ as $t\to\infty$. 
Putting the pieces together yields 
$R^{(2)}_{t}- R^{(2)}_{t-\varepsilon(t)}\stackrel{d}{\to}0$ as $t\to\infty$ which completes the verification.
\hfill$\square$
\end{proof}

\subsection{Proof of Corollary \ref{Cor1} and the proposition in Remark \ref{Levy}}

\begin{proof}
Notice that 
\begin{align*}
P\big[F_t \in A\mid X_{0}=i\big]&=\sum_{j\in S}P\big[F_t\in A, X_{t}=j \mid X_{0}=i \big]\\
&=\sum_{j\in S} P_{ij}(0,t) P\big[F_t\in A \mid X_{0}=i, X_{t}=j\big].
\end{align*}
On $\{X_0=i,X_t=j\}$, 
\begin{align*}
F_t&=\int_{0}^{\tau_{0}^{j}}d(X_{s})e^{-\int_{s}^{\tau_{0}^{j}}c(X_{r})dr}ds\\
&\quad+\sum_{k=1}^{g^{j}_{t}}e^{-\int_{\tau_{g^{j}_{t}}}^{t}c(X_r)dr}\Big(\prod_{i=k}^{g^{j}_{t}-1}
e^{-\int_{\tau^{j}_{i}}^{\tau_{i+1}^{j}}c(X_{r})dr}
\Big)\int_{\tau^{j}_{k-1}}^{\tau_{k}^{j}}d(X_{s})e^{-\int_{s}^{\tau_{k}^{j}}c(X_{r})dr}ds\\
&\quad+\int_{\tau^{j}_{g^{j}_{t}}}^{t}d(X_{s})e^{-\int_{s}^{\tau_{0}^{j}}c(X_{r})dr}ds.
\end{align*}
Therefore, upon redefining $(J^{j}_k,K^{j}_k)$ and $K^{(i,j)}_{0}$ as 
\begin{align*}
(J^{j}_{k},K^{j}_{k})&\stackrel{}{:=}\Big(e^{-\int_{\tau_{k-1}^{j}}^{\tau_{k}^{j}}c(X_{s})ds},\int_{\tau_{k-1}^{j}}^{\tau_{k}^{j}}d(X_{s})e^{-\int_{s}^{\tau_{k}^{j}}c(X_{r})dr}ds\Big),\\
K^{(i,j)}_{0}&\stackrel{}{:=}\int_{0}^{\tau_{0}^{j}}d(X_{s})e^{-\int_{s}^{\tau_{0}^{j}}c(X_{r})dr}ds,
\end{align*}
notice that, on $\{X_0=i,X_t=j\}$, 
\begin{align*}
F_t=G^{c,d}_{j}(t-\tau_{g_{t}^{j}}^{j}) +e^{-c(j)(t-\tau_{g_{t}^{j}}^{j})}\Big[\Big(\prod_{k=1}^{g_{t}^{j}}J_{k}^{j}\Big)^{}K^{(i,j)}_{0}+\sum_{k=1}^{g_{t}^{j}}\Big(\prod_{l=k+1}^{g_{t}^{j}}J^{j}_{l}\Big)^{} K^{j}_{k}\Big].
\end{align*}
Now the assertion will follow by applying the same arguments as in the proof of Theorem \ref{P1}. 
Similar arguments to those used to prove Lemma \ref{lem2} show that $F_{g^{j}_{t}}\stackrel{d}{\to} X^{*}_{j}$ as $n\to\infty$, where $X^{*}_{j}$ satisfy \eqref{x=ax+b} with $(A,B)$ as in \eqref{MP1e21cd}. Existence and uniqueness are ensured by Assumption \ref{as22}. 

For the L\'evy driven model in \eqref{levy} the statement of Remark \ref{Levy} can be proven upon minor modifications of the proof of Theorem \ref{P1} since the pairs 
$$
\Big(e^{-\int_{\tau_{i-1}^{j}}^{\tau_{i}^{j}}a(X_{r})dr}, \int_{\tau_{i-1}^{j}}^{\tau_{i}^{j}}b^{}(X_{s})e^{-\int_{s}^{\tau_{i}^{j}}a(X_{r})dr}dL_s\Big)
$$ 
form an i.i.d. sequence. This holds due to the stationary and independent increment property of L\'evy processes:
$$
L_{\tau_{i}^{j}}-L_{\tau_{i-1}^{j}}
\stackrel{d}{=} L_{\tau_{i}^{j}- \tau_{i-1}^{j}}
\stackrel{d}{=}L_{\tau_{1}^{j}- \tau_{0}^{j}}
\stackrel{d}{=}L_{\tau_{1}^{j}}-L_{\tau_{0}^{j}}.
$$
Similarly for the added continuity factor on $\{X_{0}=i,X_{t}=j\}$ one has 
\begin{align*}
\int_{\tau_{g_{t}^{j}}^{t}}^{t}b(X_{s})e^{\int_{s}^{t}a(X_{r})dr}dL_{s}=b(j)\int_{\tau_{g_{t}^{j}}^{t}}^{t}e^{-a(j)(t-s)}dL_{s}\stackrel{d}{=}b(j)\int_{0}^{t-\tau_{g_{t}^{j}}^{j}}e^{-a(j)(t- \tau_{g_{t}^{j}}^{j} -y)}dL_{y}
\end{align*}
where last equality follows by observing $L_{s}-L_{\tau_{g_{t}^{j}}^{j}}\stackrel{d}{=}L_{s-\tau_{g_{t}^{j}}^{j}}.$ 
\hfill $\square$
\end{proof}

\subsection{Proof of Theorem \ref{T2}(a)}
\begin{proof}
Let 
$\widetilde{Y}_{t}=Y_{0}+\int_{0}^{t}b(X_{s})\mb{\Phi}^{-1}(0,s)dW_{s}$
and note that $Y_{t}=\mb{\Phi}(0,t)\widetilde{Y}_{t}$.
For any $A\in\mathcal{B}(\R)$,
\begin{align*}
P\big[|Y_{t}|^{\frac{1}{\sqrt{t}}}\in A\mid X_{0}=i, Y_{0}=y_{0}\big] 
&=\sum_{j\in S}P\big[|Y_{t}|^{\frac{1}{\sqrt{t}}}\in A, X_{t}=j\mid X_{0}=i, Y_{0}=y_{0}\big]\\
&=\sum_{j\in S}P_{ij}(0,t)P\big[|Y_{t}|^{\frac{1}{\sqrt{t}}}\in A\mid X_{0}=i, X_{t}=j, Y_{0}=y_{0}\big]\\
&=\sum_{j\in S}P_{ij}(0,t)P\big[\mb{\Phi}(0,t)^{\frac{1}{\sqrt{t}}} |\widetilde{Y}_{t}|^{\frac{1}{\sqrt{t}}}\in A\mid X_{0}=i, X_{t}=j, Y_{0}=y_{0}\big].
\end{align*}
We will let $t\to\infty$ and determine the limit. 
On $\{X_{0}=i,X_{t}=j\}$, 
$$
\mb{\Phi}(0,t)=e^{-a(j)\big(t-\tau^{j}_{g^{j}_{t}}\big)}\Big(\prod_{l=1}^{g_{t}^{j}} J^{j}_{l}\Big)J_{0}^{(i,j)}.
$$
Therefore, on $\{X_{0}=i,X_{t}=j\}$, 
\beqn
\frac{\mb{\Phi}(0,t)^{\frac{1}{\sqrt{t}}}}{e^{^{-\sqrt{t}E_{\pi}a(\cdot)}}}=
\exp\Big\{\frac{\log J_{0}^{(i,j)}}{\sqrt{t}}-a(j)\frac{t-\tau_{g_{t}^{j}}^{j}}{\sqrt{t}}\Big\} 
\exp\Big\{\frac{\sum_{l=1}^{g_{t}^{j}}\log J_{l}^{j}}{\sqrt{t}}+\sqrt{t}E_{\pi}a(\cdot)\Big\}\label{tre1}
\eeqn
Since both $\log J_{0}^{(i,j)}$ and $t-\tau_{g_{t}^{j}}^{j}$ are $O_{P}(1)$ the first factor in the product on the right-hand side goes to $1$ in probability as $t\to\infty$. 
Note that, by the ergodic theorem, $E\log J_{i}^{j}=-E|I_{j}|E_{\pi}a(\cdot)$.  By the renewal theorem, 
$$
\frac{g^{j}_{t}}{t_{}}\stackrel{\text{a.s.}}{\to} \frac{1}{E|I^{j}_1|}\quad\text{as } t\to\infty
$$
and the second factor in the product on the right-hand side of \eqref{tre1} satisfies
\beqn
e^{M_{t}}:=\exp\Big\{\frac{\sum_{l=1}^{g_{t_{}}^{j}}\log J_{l}^{j}}{\sqrt{t_{}}}+\sqrt{t_{}}E_{\pi}a(\cdot)\Big\}\stackrel{d}{\to}
\exp\bigg\{\frac{\sigma_{j}}{\sqrt{E|I^{j}_1|}}N\bigg\}\quad\text{as } n\to\infty,\label{tre11}
\eeqn
where $N\sim\normal(0,1)$, by applying the central limit theorem for the renewal reward process (when the reward is $\log J_{l}^{j}$ for the $l$-th interval $I^{j}_{l}$) (see Theorem 2.2.5 of \cite{tijms2003first}) together with the continuous mapping theorem using the map $x\mapsto e^{x}$. 
 
We now consider the conditional distribution of $|\widetilde{Y}_t|$ appearing in the product $\mb{\Phi}(0,t)^{\frac{1}{\sqrt{t}}} |\widetilde{Y}_{t}|^{\frac{1}{\sqrt{t}}}$. 
\begin{align*}
&P\Big[|\widetilde{Y}_t|\in \cdot\mid X_{0}=i, X_{t}=j, Y_{0}=y_{0}\Big]\\
&\quad=P\Big[\Big|y_0+\Big(\mb{\Phi}^{-2}(0,t)\int_{0}^{t}b^{2}(X_{s})\mb{\Phi}^{2}(s,t)ds\Big)^{1/2}N\Big|\in\cdot\mid X_{0}=i, X_{t}=j\Big]\\
&\quad=P\Big[\Big|y_0+\Big(\frac{G^{2a,b^2}_{j}(t-\tau_{g_{t}^{j}}^{j})}{\mb{\Phi}^{2}(0,t)}+\frac{K_{0}^{(i,j)}+S_{g_{t}^{j}}^{*}}{\big(J_{0}^{(i,j)}\big)^{2}}\Big)^{1/2}N\Big|\in\cdot\mid X_{0}=i, X_{t}=j\Big],
\end{align*} 
where $N\sim\normal(0,1)$, $N\indep X$ and 
$$
S^{*}_{n}:=\sum_{i=1}^{n}\frac{K^{j}_{i}}{(J^{j}_{i})^2}\prod_{k=1}^{i-1}\Big(\frac{1}{J^{j}_{k}}\Big)^{2}, \quad n\ge 1.
$$ 
Applying similar arguments as those in Step 4 in the proof of Lemma \ref{lem2}, using 
$$
E\log\frac{1}{(J^{j}_{1})^2}<0, \quad E\log^{+}\frac{K^{j}_1}{(J^{j}_{1})^2}<\infty,
$$
we find that 
\begin{align*}
S^{*}_{g^{j}_{t}} \stackrel{d}{\to} S^{*} \quad\text{as } n\to\infty, \quad 
S^{*}\stackrel{d}{=} \frac{1}{(J^{j}_{1})^2}S^{*}+\frac{K^{j}_{1}}{(J^{j}_{1})^2},\quad 
S^{*}\indep \Big(\frac{1}{J^{j}_{1}},\frac{K^{j}_{1}}{(J^{j}_{1})^2}\Big).
\end{align*}
In the transient regime, 
$$
\frac{\log|\mb{\Phi}^{2}(0,t)|}{t}\stackrel{\text{a.s.}}{\to}-2E_{\pi}a(\cdot)>0\quad \text{as } t\to\infty,
$$
so $\mb{\Phi}^{}(0,t)$ diverges exponentially as $t\to\infty$, and since $G^{2a,b^2}_{j}(t-\tau_{g^{j}_{t}}^{j})$ is $O_{P}(1)$, 
applying similar arguments as those proving \eqref{condition2} shows that 
\begin{align}
P\Big[\log |\widetilde{Y}_{t_{}}|\in \cdot \mid X_{0}=i\Big]\stackrel{w}{\to} 
P\Big[\log\Big| y_{0}+\Big(\frac{K_{0}^{(i,j)}+S_{}^{*}}{(J_{0}^{(i,j)})^{2}}\Big)\Big|\in \cdot\Big]
\quad\text{as }t\to\infty.\label{tre12}
\end{align}
Since $P\big[(J^{j}_{1})^{-1}=0\big]=0$, applying Theorem 1.3 of \cite{alsmeyer2009distributional} shows that the distribution of $S^{*}$ is absolutely continuous. From \eqref{tre12} follows that 
$$
\frac{\log|\widetilde{Y}_{t_{}}|}{\sqrt{t_{}}} \stackrel{P}{\to} 0 \quad \text{as } t\to \infty
$$
which implies, using Lemma \ref{lem1.5}, that
$$
P\Big[|\widetilde{Y}_t^{\frac{1}{\sqrt{t}}}| \in\cdot \mid X_0=i,X_{t}=j\Big]\stackrel{w}{\to} \delta_{1}(\cdot)\quad\text{as }t\to\infty.
$$
Combining \eqref{tre11} and the consequence of \eqref{tre12} gives 
\begin{align}
P\Big[\frac{|Y_{t_{}}|^{\frac{1}{\sqrt{t_{}}}}}{e^{^{-\sqrt{t_{}}E_{\pi}a(\cdot)}}}\in\cdot\mid X_{0}=i, Y_{0}=y_{0}\Big]\stackrel{w}{\to}
P\Big[\exp\Big\{\frac{\sigma_{j}}{\sqrt{E|I^{j}_{1}|}}N\Big\}\in\cdot\Big]
\quad\text{as }t\to\infty,\label{tre14}
\end{align} 
Therefore, using similar arguments as those used in Lemma \ref{lem1.5}, taking 
$$
L_{t}^{(1)}:=0,\quad L_{t}^{(2)}:=\exp\Big\{\frac{\log J_{0}^{(i,j)}}{\sqrt{t}}-a(j)\frac{t-\tau_{g_{t}^{j}}^{j}}{\sqrt{t}}+\frac{\log|\tilde{Y}_{t}|}{\sqrt{t}}\Big\}, \quad\text{and}\quad R_{t}:=e^{M_{t}}
$$ 
in the statement of Lemma \ref{lem1.5}, \eqref{tre14} implies 
\begin{align}
P\Big[\frac{|Y_{t}|^{\frac{1}{\sqrt{t}}}}{e^{^{-\sqrt{t}E_{\pi}a(\cdot)}}}\in\cdot\mid X_{0}=i,X_{t}=j,Y_{0}=y_{0}\Big]
\stackrel{w}{\to}
P\Big[\exp\Big\{\frac{\sigma_{j}}{\sqrt{E|I^{j}_{1}|}}N\Big\}\in\cdot\Big]
\quad\text{as }t\to\infty\label{tre13}
\end{align} 
given the conditions of Lemma \ref{lem1.5} hold for $e^{M}$ or $M$. That is, given that 
\beqn
M_{t}-M_{t-\varepsilon(t)}\stackrel{P}{\to}0 \quad\text{as } t\to\infty \label{Mverific}
\eeqn 
for some increasing function $\varepsilon(t)$ such that $\frac{\varepsilon(t)}{t}\to 0$ as $t\to\infty$. 
The verification of these conditions is done in Subsection \ref{T2ver}. Using \eqref{tre13}
along with $\lim_{t\to\infty}P_{ij}(0,t)=\pi_{j}$ in the following display for any $x\in\R$
\begin{align*}
&P\Big[\frac{1}{\sqrt{t}}\log|Y_{t}| +\sqrt{t}E_{\pi}a(\cdot)\in (x,\infty)\mid X_{0}=i,Y_{0}=y_{0}\Big]\\
&\quad=\sum_{j\in S}P_{ij}(0,t)P\Big[\frac{1}{\sqrt{t}}\log|Y_{t}| +\sqrt{t}E_{\pi}a(\cdot)\in (x,\infty)\mid X_{0}=i, X_{t}=j, Y_{0}=y_{0}\Big]
\end{align*}
the assertion of Theorem \ref{T2}(a) follows by taking the limit of the above expression as $t\to\infty$.
\hfill$\square$
\end{proof}

\subsubsection{Verification that $M_{t}$ in \eqref{tre11} satisfies \eqref{condition2.01} of Lemma \ref{lem1.5}}\label{T2ver}
All conditions hold trivially except \eqref{Mverific} which is proved below.

\begin{proof} 
Take $\varepsilon(t):=\sqrt{t}$. 
Define $\widetilde{M}_{t} :=\frac{\sum_{i=1}^{g_{t}^{j}}Y^{*}_{i}}{\sqrt{t}}$, where $Y^{*}_{k}:=\log J^{j}_{k}-E\log J^{j}_{k}$ for $k=1,\ldots,g_{t}^{j}$ is a sequence of random variables with zero means. Note that it is sufficient to prove condition \eqref{condition2.01} for $\widetilde{M}_{t}$ since 
$$
M_{t}=\widetilde{M}_{t}- \frac{E\tau_{0}+(t-E\tau_{g_{t}})E_{\pi}a(\cdot)}{\sqrt{t}}
$$ 
and the second term in the right-hand side above is $o_{_{P}}(1).$ Now observe 
$$
\widetilde{M}_{t}-\widetilde{M}_{t-\sqrt{t}}
=\frac{\sum_{i=1}^{g^{j}_{t-\sqrt{t}}}Y_{i}}{\sqrt{t}}\Big[1-\Big(1-\frac{1}{\sqrt{t}}\Big)^{-\frac{1}{2}}\Big]+\frac{\sum_{i=g^{j}_{t-\sqrt{t}}+1}^{g^{j}_{t}}Y_{i}}{\sqrt{t}}.
$$
Using the asymptotic approximation 
$$
\big(1-\frac{1}{\sqrt{t}}\big)^{-\frac{1}{2}}=1+\frac{1}{2\sqrt{t}}+O(t^{-1})
$$ 
shows that
$$
\frac{\sum_{i=1}^{g^{j}_{t-\sqrt{t}}}Y_{i}}{\sqrt{t}}\Big[1-\big(1-\frac{1}{\sqrt{t}}\big)^{-\frac{1}{2}}\Big]
=\frac{\sum_{i=1}^{g^{j}_{t-\sqrt{t}}}Y_{i}}{2t}+o_{_P}(t^{-\frac{1}{2}})\stackrel{P}{\to}EY^{*}_{k}=0
$$ 
by the renewal version of the law of large numbers. By a similar argument 
$$
\frac{\sum_{i=g^{j}_{t-\sqrt{t}}+1}^{g^{j}_{t}}Y^{*}_{i}}{\sqrt{t}}
=\frac{\sum_{i=g^{j}_{t-\sqrt{t}}+1}^{g^{j}_{t}}Y^{*}_{i}}{g^{j}_{t} - g^{j}_{t-\sqrt{t}}}\frac{g^{j}_{t} - g^{j}_{t-\sqrt{t}}}{\sqrt{t}} \stackrel{\text{a.s.}}{\to}\frac{1}{E|I_{j}|} EY^{*}_{k}=0.
$$ 
The verification is complete.
\hfill$\square$
\end{proof}

\subsection{Proof of Theorem \ref{T2}(b)}

Before proving Theorem \ref{T2}(b) we introduce some notation useful for random walk related arguments. Fix $j\in S$. Under $E_{\pi}a(\cdot)=0$, the one dimensional random walk $S^{\otimes j}_{n}:=\sum_{k=1}^{n}\log J_{k}^{j}$, $S^{\otimes j}_{0}=0$, is null-recurrent. Recall that $\log J_{k}^{j}=-\int_{\tau^{j}_{k-1}}^{\tau^{j}_{k}}a(X_{s})ds$ with mean $-E_{\pi}a(\cdot)E|I^{j}_1|=0$ and variance $\sigma^{2}_{j}$. For this random walk we define the sequence $(Z^{\otimes j}_{k})_{k\ge 1}$ of ladder variables as follows. Define the sequence $(T^{\otimes j}_{k})_{k\ge 1}$ such that 
$$
T^{\otimes j}_{1}:=\inf\{m:S^{\otimes}_{m}>0\},\quad T^{\otimes j}_{k}:=\inf\{m>T^{\otimes j}_{k-1}: S^{\otimes j}_{m}>S^{\otimes j}_{T_{k-1}}\}.
$$
The ladder variables are defined as $Z^{\otimes j}_{k}:=S^{\otimes j}_{T^{\otimes j}_{k}} - S^{\otimes j}_{T^{\otimes j}_{k-1}}$.
Let $M^{\otimes j}_{n}:=\max_{1\le k \le n}\sum_{i=1}^{k}\log J_{i}^{j}$.

\begin{proof}
We begin by expanding \eqref{ito1} differently from what was done in the proof of Theorem \ref{T2}(a). Since $E_{\pi}a(\cdot)=0$, instead of extracting $\mb{\Phi}(0,t)$ as in the proof of Theorem \ref{T2}(a), we will extract $Z_{n}^{*}:=\max_{1\le k\le n} \Big(\prod_{i=1}^{k-1}J^{j}_{i}\Big)^{2}K^{j}_{k}$, for $n=g_{t}^{j}.$ Note that all $J_{k}^{j},K_{k}^{j}$ are non-negative.
Let $N\sim \normal(0,1)$ and recall notations $Q^{(1)}_t$ and $Q^{(2)}_t$ in \eqref{z_1z_2} and $S_n:=\sum_{k=1}^n\big(\prod_{i=1}^{k-1}J^{j}_i\big)^2K^{j}_k$.
Using that
$\widetilde{S}_{g_{t}^{j}}\stackrel{d}{=}S_{g_{t}^{j}}$ for all $t>0$ (Step 1 in the proof of \eqref{toprove3Xj}), on $\{X_{0}=i, X_{t}=j\}$, 
\begin{align}
Y_{t}&\stackrel{d}{=} Q^{(1)}_{t}+\sqrt{Q^{(2)}_{t}}N \non \\
&\stackrel{d}{=} 
e^{ -a(j)(t-\tau_{g_{t}^{j}}^{j})}\Big(\prod_{k=1}^{g^{j}_{t}}J^{j}_{k}\Big)J^{(i,j)}_{0}
+\bigg(
G^{2a,b^2}_{j}\big(t-\tau_{g_{t}^{j}}^{j}\big)+e^{-2a(j)(t-\tau_{g_{t}^{j}}^{j})}\Big[\Big(\prod_{k=1}^{g_{t}^{j}-1}J^{j}_{k}\Big)^{2} K^{(i,j)}_{0}+ \widetilde{S}_{g_{t}^{j}}\Big]
\bigg)^{1/2}N \non\\
&\stackrel{d}{=}Q^{(1)}_{t}
+\bigg(
G^{2a,b^2}_{j}\big(t-\tau_{g_{t}^{j}}^{j}\big)+e^{-2a(j)(t-\tau_{g_{t}^{j}}^{j})}\Big[\Big(\prod_{k=1}^{g_{t}^{j}-1}J^{j}_{k}\Big)^{2} K^{(i,j)}_{0}+ S_{g_{t}^{j}}\Big]
\bigg)^{1/2}N \non\\ 
&=:\sqrt{Z^{*}_{g_{t}^{j}}}\bigg(\mb{A}^{(1)}_{t}
+\bigg(
\mb{A}^{(2)}_{t}+ \mb{A}^{(3)}_{t}\frac{S_{g^{j}_{t}}}{Z^{*}_{g_{t}^{j}}}
\bigg)^{1/2}N\bigg), \label{2be1}
\end{align}
where 
\begin{align*}
\mb{A}^{(1)}_{t}&:=e^{-a(j)(t-\tau_{g_{t}^{j}}^{j})}J_{0}^{(i,j)}\frac{\prod_{k=1}^{g^{j}_{t}}J^{j}_{k}}{\sqrt{Z^{*}_{g_{t}^{j}}}}, \\
\mb{A}^{(2)}_{t}&:= \frac{G^{2a,b^2}_{j}\big(t-\tau_{g_{t}^{j}}^{j}\big)+e^{-2a(j)(t-\tau_{g_{t}^{j}}^{j})}\Big(\prod_{k=1}^{g_{t}^{j}-1}J^{j}_{k}\Big)^{2} K^{(i,j)}_{0}}{Z^{*}_{g_{t}^{j}}} ,\\
\mb{A}^{(3)}_{t}&:=e^{-2a(j)(t-\tau_{g_{t}^{j}}^{j})}. 
\end{align*}
Since $K_{n}^{j}\big(\prod_{i=1}^{n-1}J^{j}_{i}\big)^{2}\le Z^{*}_{n}$ for all $n\ge 1$, $\mb{A}^{(1)}_{\cdot},\mb{A}^{(2)}_{\cdot},\mb{A}^{(3)}_{\cdot}$ are all $O_{P}(1)$. 
Denote the second factor in the product \eqref{2be1} by $\mb{A}^{}_{t}$. 
Since $Z^{*}_{n}\le S_{n}\le nZ^{*}_n$ for all $n\ge 1$, $|\mb{A}^{}_{t}|=O_{P}(\sqrt{t})$ and is strictly positive. 
Moreover, 
$$
|\mb{A}_{t}|^{\frac{1}{\sqrt{t}}}
=\exp\bigg\{\frac{1}{\sqrt{t}}\log\bigg|\mb{A}^{(1)}_{t}+\bigg(\mb{A}^{(2)}_{t}+ \mb{A}^{(3)}_{t}\frac{S_{g^{j}_{t}}}{Z^{*}_{g_{t}^{j}}}\bigg)^{1/2}N\bigg|\bigg\}\stackrel{P}{\to} 1
\quad\text{as }t\to\infty.
$$ 
Consider the probability
\begin{align}
P\big[|Y_{t}|^{\frac{1}{\sqrt{t}}}\in A\mid X_{0}=i,Y_{0}=y_{0}\big]
=\sum_{j\in S}P_{ij}(0,t)P\big[\big(Z^{*}_{g^{j}_{t}}\big)^{\frac{1}{2\sqrt{t}}} |\mb{A}_{t}|^{\frac{1}{\sqrt{t}}}\in A\mid X_{0}=i, X_{t}=j, Y_{0}=y_{0}\big].
\label{t2e2}
\end{align}
Applying Slutsky's theorem through Lemma \ref{lem1.5}, by taking 
$$
\big(L^{(1)}_{t},L^{(2)}_{t},R_{t}\big):=\Big(0, |\mb{A}_{t}|^{\frac{1}{\sqrt{t}}},\big(Z^{*}_{g^{j}_{t}}\big)^{\frac{1}{2\sqrt{t}}}\Big)
$$ 
in the statement of the lemma in order to get rid of $X_{t}=j$ in the conditioning event on the right-hand side in \eqref{t2e2},  yields the following limit result 
\begin{align}
\lim_{t\to\infty}P\big[|Y_{t}|^{\frac{1}{\sqrt{t_{}}}}\in A \big|  X_{0}=i,Y_{0}=y_{0}\big]
=\sum_{j\in S}\pi_{j}\lim_{t\to\infty}P\big[\big(Z^{*}_{g^{j}_{t_{}}}\big)^{\frac{1}{2\sqrt{t_{}}}} \in A\mid X_{0}=i\big]
\label{t2be3}
\end{align}
given the right hand limit of \eqref{t2be3} exists and $\frac{\log Z_{g^{j}_{t_{}}}^{*}}{\sqrt{t_{}}}$ satisfies condition \eqref{Mverific} as $M_{t}$. 
Lemma \ref{lem3} below verifies that. Immediately from the definitions of $M^{\otimes j}_{n}$ and $Z_{n}^{*}$ (see also (4.4) in \cite{hitczenko2011renorming}) follows
$$
2M^{\otimes j}_{g_{t}^{j}}-\max_{1\le k\le g_{t}^{j}}\log K_{k}^{j}\le \log Z_{g_{t}^{j}}^{*} \le 2M^{\otimes j}_{g_{t}^{j}}+\max_{1\le k\le g_{t}^{j}}\log K_{k}^{j}.
$$
It is shown in \cite{hitczenko2011renorming} that 
$$
\frac{M^{\otimes j}_{n}}{\sqrt{n}}\stackrel{d}{\to}\sigma_{j}|N| \quad\text{and}\quad 
\frac{\max_{1\le k\le n}\log K_{k}^{j}}{\sqrt{n}}\stackrel{P}{\to}0
\quad\text{as } n\to\infty.
$$
However, we need a version of the above convergence in a renewal time context: Lemma \ref{lem3} below. 
Using this lemma, it follows that 
$$
\frac{\log Z_{g^{j}_{t_{}}}^{*}}{2\sqrt{t_{}}}\stackrel{d}{\to}\sigma_{j}|N|
\quad\text{as }t\to\infty,
$$ 
which concludes the proof of Theorem \ref{T2}(b).
\hfill$\square$
\end{proof}

\begin{lemma}\label{lem3}
Under the assumptions of Theorem \ref{T2}(b), as $t\to\infty$, 
\begin{align*}
\text{(a)}\quad\frac{\max_{1\le k\le g_{t}^{j}}\log K^{j}_{g_{t}^{j}}}{\sqrt{g_{t}^{j}}}\stackrel{\text{a.s.}}{\to}0,
\quad 
\text{(b)}\quad\frac{M^{\otimes j}_{g_{t}^{j}}}{\sqrt{g_{t}^{j}}}\stackrel{d}{\rightarrow} \sigma_{j} |N|, 
\quad 
\text{(c)}\quad\frac{\log Z_{g^{j}_{t_{}}}^{*}}{\sqrt{t_{}}} \quad\text{verifies}\,\,\eqref{Mverific}\,\,\text{as} \,\, M_{t}.
\end{align*}
\end{lemma}

\begin{proof} 
The terms of $(\log K^{j}_{n})_{n\ge 1}$ are i.i.d. Since, by assumption, $E\big(\log K_{1}^{j}\big)^{2}<\infty$, the strong law of large numbers gives
$$
\frac{1}{n}\sum_{k=1}^n \big(\log K_{k}^{j}\big)^{2}\stackrel{\text{a.s.}}{\to}E\big(\log K_{1}^{j}\big)^{2}<\infty 
\quad \text{as }n\to\infty
$$
or equivalently, by the Borel-Cantelli lemma, as used in Step 2 in the proof of Lemma \ref{lem2}, 
$$
\frac{\max_{1\le k \le n}\log K^{j}_{k}}{\sqrt{n}}\stackrel{\text{a.s.}}{\to}0 \quad n\to\infty. 
$$ 
By Theorem 2.1 in \cite{gut2009stopped} we may replace $n$ by $g_{t}^{j}$, from which the conclusion of part (a) follows. We will prove part (b) and (c). For readability, write $\mu_{j}:=1/E|I^{j}_{1}|$. 
By Theorem I in \cite{ErdosKac1946},  
\beqn
\frac{M^{\otimes j}_n}{\sqrt{n}}\stackrel{d}{\to} \sigma_{j}|N| \quad\text{as } n\to\infty
\label{erdosKac}
\eeqn
as $M^{\otimes j}_n$ is the maximum partial sum of a random walk with mean $0$. 

We prove part (b) by showing Anscombe's condition 
(uniform continuity in probability, condition (A) on page 16 in \cite{gut2009stopped}) 
which in our context requires showing the following: Given $\gamma>0,\eta>0$, there exist $\delta>0,n_{0}>0$ such that 
\beqn
P\Big[\max_{\{k: |k-n|<n\delta\}}\Big|\frac{M^{\otimes j}_{k}}{\sqrt{k}}-\frac{M^{\otimes j}_{n}}{\sqrt{n}}\Big|> \gamma \Big]<\eta\quad \forall \,\, n\ge n_{0}.\label{cond0.001}
\eeqn 
Since $\{k: |k-n|<n\delta\}=\{n\le k\le n(1+\delta)\}\cup\{n(1-\delta)\le k\le n\}$ an upper bound for the probability in \eqref{cond0.001}  is 
$$P\Big[\max_{\{k: n\le k\le n+n\delta\}}\Big|\frac{M^{\otimes j}_{k}}{\sqrt{k}}-\frac{M^{\otimes j}_{n}}{\sqrt{n}}\Big|> \gamma \Big]+  P\Big[\max_{\{k: n(1-\delta)\le k \le n\}}\Big|\frac{M^{\otimes j}_{k}}{\sqrt{k}}-\frac{M^{\otimes j}_{n}}{\sqrt{n}}\Big|> \gamma \Big].
$$
We prove that the first term is smaller than $\eta/2$. Similar arguments show that second term is smaller than $\eta/2$ from which \eqref{cond0.001} follows.
Observe that
\begin{align}
&P\Big[\max_{\{k: n\le k\le n(1+\delta)\}}\bigg|\frac{M^{\otimes j}_{k}}{\sqrt{k}}-\frac{M^{\otimes j}_{n}}{\sqrt{n}}\Big|> \gamma \bigg] 
\le P\Big[\max_{\{k: n\le k\le n(1+\delta)\}}\Big(\Big|\frac{M^{\otimes j}_{k}}{\sqrt{k}}-\frac{M^{\otimes j}_{k}}{\sqrt{n}}\Big|+\Big|\frac{M^{\otimes j}_{k}}{\sqrt{n}}- \frac{M^{\otimes j}_{n}}{\sqrt{n}}\Big|\Big)> \gamma \Big]\non\\
&\quad\le P\Big[\frac{M^{\otimes j}_{n(1+\delta)}}{\sqrt{n}}[(1+\delta)^{\frac{1}{2}}-1]>\frac{\gamma}{2}\Big]+P\Big[\frac{M^{\otimes j}_{n(1+\delta)}-M^{\otimes j}_{n}}{\sqrt{n}}>\frac{\gamma}{2}\Big].\label{cond0.003}
\end{align}
For all $\delta\ge 0$, $(1+\delta)^{\frac{1}{2}}\le (1+\frac{\delta}{2})$. Consequently, for all $\delta\ge 0$,
$$
P\Big[\frac{M^{\otimes j}_{n(1+\delta)}}{\sqrt{n}}[(1+\delta)^{\frac{1}{2}}-1]>\frac{\gamma}{2}\Big]
\le P\Big[\frac{\delta(1+\frac{\delta}{2})}{2} \frac{M^{\otimes j}_{n(1+\delta)}}{\sqrt{n(1+\delta)}}>\frac{\gamma}{2}\Big].
$$
For given $\gamma>0,\eta>0$ one can choose a small $\delta_{\eta}>0$ and a large $n_{0}^{(1)}$ such that for all $\delta\le \delta^{(1)}_{\eta}$, as a consequence of \eqref{erdosKac},  
$$
P\Big[\frac{\delta(1+\frac{\delta}{2})}{2} \frac{M^{\otimes j}_{n(1+\delta)}}{\sqrt{n(1+\delta)}}>\frac{\gamma}{2}\Big]<\frac{\eta}{4}, \quad \text{for } n\ge n_{0}^{(1)},
$$
The second term in \eqref{cond0.003} can be shown to be smaller than $\eta/4$. This can be shown by first writing, with  
$A_{1}:=\{(x,y): 0\le y<\infty,-\infty<x<y\}$, 
\begin{align}
P\Big[M^{\otimes j}_{n(1+\delta)}-M^{\otimes j}_{n}> \frac{\gamma}{2} \sqrt{n}\Big] 
&=\int_{A_{1}}P\big[M^{\otimes j}_{n(1+\delta)}-M^{\otimes j}_{n}>\frac{\gamma}{2}\sqrt{n}\mid (S^{\otimes j}_{n},M^{\otimes j}_{n})=(s,t)\big]f_{(S^{\otimes j}_{n},M^{\otimes j}_{n})}(s,t)dsdt\non\\
&=\int_{A_{1}}P\Big[\frac{M^{\otimes j}_{n\delta}}{\sqrt{n\delta}}>\frac{t-s}{\sqrt{n\delta}}+\frac{\gamma\sqrt{n}}{2\sqrt{n\delta}}\Big]f_{(S^{\otimes j}_{n},M^{\otimes j}_{n})}(s,t)dsdt, \label{cond0.004}
\end{align}
where \eqref{cond0.004} follows from the independent increment property of the random walk $(S_{n}^{\otimes j})_{n\ge 1}$. 
For $(s,t)\in A_{1}$, 
$$
P\Big[\frac{M^{\otimes j}_{n\delta}}{\sqrt{n\delta}}>\frac{t-s}{\sqrt{n\delta}}+\frac{\gamma\sqrt{n}}{2\sqrt{n\delta}}\Big]\le P\Big[\frac{M^{\otimes j}_{n\delta}}{\sqrt{n\delta}}> \frac{\gamma}{2\sqrt{\delta}}\Big].
$$ 
By choosing $\delta^{(2)}_{\eta,\gamma}>0$ and $\delta<\delta^{(2)}_{\eta,\gamma}$ such that $\gamma/(2\sqrt{\delta})$ is sufficiently large and choosing $n_{\eta,\gamma}^{(2)}$ large, 
as a consequence of \eqref{erdosKac}, 
$$
P\Big[\frac{M^{\otimes j}_{n\delta}}{\sqrt{n\delta}}> \frac{\gamma}{2\sqrt{\delta}}\Big]<\frac{\eta}{4}, 
\quad \text{for } n\ge n_{\eta,\gamma}^{(2)}.
$$
Putting this upper bound in \eqref{cond0.004}, part (b) follows by taking $n\ge n_{0}^{(1)} \vee n_{\eta,\gamma}^{(2)}$ and $\delta\le \delta^{(1)}_{\eta}\wedge \delta^{(2)}_{\eta,\gamma}$.

Part (c) follows if  $M^{\otimes j}_{g^j_{t}}/\sqrt{t}$ verifies \eqref{Mverific} as $M_{t}$. 
Take $t\mapsto \varepsilon(t)$ to be an increasing function satisfying $\lim_{t\to\infty}\varepsilon(t)=\infty$ and $\lim_{t\to\infty}\varepsilon(t)/t=0$. 
Observe that
\begin{align}
\frac{M^{\otimes j}_{g^j_{t}}}{\sqrt{t}}-\frac{M^{\otimes j}_{g^j_{t-\varepsilon(t)}}}{\sqrt{t-\varepsilon(t)}}=\frac{M^{\otimes j}_{g^j_{t}}-M^{\otimes j}_{[\mu_{j}t]}}{\sqrt{t}} +\frac{M^{\otimes j}_{[\mu_{j}t]}}{\sqrt{t}}  - \frac{M^{\otimes j}_{[\mu_{j}(t-\varepsilon(t))]}}{\sqrt{t-\varepsilon(t)}}+ \frac{M^{\otimes j}_{[\mu_{j}(t-\varepsilon(t))]}-M^{\otimes j}_{g^j_{t-\varepsilon(t)}}}{\sqrt{t-\varepsilon(t)}}.\label{cond0.005}  \end{align}
For $\delta>0,$ set $A^{}_{(\delta,j)}:=\{n\in\mathbb{N}:| n/t-\mu_{j}|\le \delta \}$ and observe that 
$\lim_{t\to\infty}P[g_{t}\in A^{}_{(\delta,j)}]=1$. For the first term on the right-hand side in \eqref{cond0.005}, 
\begin{align*}
P\Big[\frac{M^{\otimes j}_{g^j_{t}}-M^{\otimes j}_{[\mu_{j}t]}}{\sqrt{t}}>\gamma\Big]
&\le P\Big[\frac{M^{\otimes j}_{g^j_{t}}-M^{\otimes j}_{[\mu_{j}t]}}{\sqrt{t}} >\gamma, g^j_{t}\in A^{}_{(\delta,j)}\Big]+P\Big[\frac{M^{\otimes j}_{g^j_{t}}-M^{\otimes j}_{[\mu_{j}t]}}{\sqrt{t}} >\gamma, g^j_{t}\in A^{c}_{(\delta,j)}\Big]\\
&\le P\Big[\frac{M^{\otimes j}_{[(\mu_{j}+\delta)t]}-M^{\otimes j}_{[\mu_{j}t]}}{\sqrt{t}} >\gamma\Big]
+P\Big[\frac{M^{\otimes j}_{[\mu_{j}t]}-M^{\otimes j}_{[(\mu_{j}-\delta)t]}}{\sqrt{t}} >\gamma\Big]
+P\Big[g^j_{t}\notin A^{c}_{(\delta,j)}\Big].
\end{align*}
The first two terms on the right-hand side above goes to $0$ as $t\to\infty$ from the same arguments as those applied to prove part (b). The third term goes to $0$ as $t\to\infty$ since $g^j_t/t\stackrel{\text{a.s.}}{\to}\mu_j$ as $t\to\infty$.
Consequently, the first term on the right-hand side in \eqref{cond0.005} goes to $0$ in probability as $t\to\infty$. 
Similarly, the third term on the right-hand side in \eqref{cond0.005} goes to $0$ in probability as $t\to\infty$ since $t-\varepsilon(t)\to\infty$ as $t\to\infty$. 
The second term on the right-hand side in \eqref{cond0.005} goes to $0$ in probability as $t\to\infty$ by using an argument similar to \eqref{cond0.003} with $\varepsilon(t)/t\to 0$ as $t\to\infty$. 
This concludes the proof of part (c).
\hfill$\square$
\end{proof}

\subsection{Proof of Proposition \ref{CIRapp}}
\begin{proof}
Let $(U^{i}_t)_{t\ge 0}$, $i=1,\dots,n$ be Ornstein-Uhlenbeck processes under a Markovian environment given by 
$dU^{i}_{t}=-a(X_{t})U^{i}_{t}dt+b(X_{t})dW^{i}_{t}$, where $W^{1},\dots,W^{n}$ are i.i.d. standard Brownian motions on $\R$. 
Using Ito's lemma we may express $R_{t}:=\sum_{i=1}^{n}(U_{t}^{i})^{2}$ as the solution to a stochastic differential equation: 
\begin{align*}
d R_{t}^{} 
&= \sum_{i=1}^{n}d(U_{t}^{i})^{2}\\
&=\sum_{i=1}^{n} \Big[2 U_{t}^{i}dU_{t}^{i}+ 2 d\langle U^{i}\rangle_{t}\Big]\\
&=\sum_{i=1}^{n}\Big(-2a(X_{t}) (U_{t}^{i})^{2}+b^{2}(X_{t})\Big) dt + \sum_{i=1}^{n} 2b(X_{t})U_{t}^{i}dW_{t}^{i}\\
&=2\alpha(X_{t})\Big[\frac{nb^{2}(X_{t})}{2a(X_{t})} - R_{t}\Big]dt + 2b(X_{t}) \sqrt{R_{t}}\sum_{i=1}^{n} \frac{U_{t}^{i}dW_{t}^{i}}{\sqrt{R_{t}}}\\
&=\kappa(X_{t})(\theta(X_{t}) - R_{t})dt+\xi(X_{t}) \sqrt{R_{t}} \sum_{i=1}^{n} \frac{U_{t}^{i}dW_{t}^{i}}{\sqrt{R_{t}}}
\end{align*}
Since $X \indep  (W^{i})_{i=1}^{n}$ and $(M_t)_{t\ge 0}$ given by $M_t:=\int_{0}^{t}\sum_{i=1}^{n}U_{s}^{i}dW_{s}^{i}$ is a martingale with quadratic variation $\langle M\rangle_t = \int_{0}^{t}\sum_{i=1}^{n}(U_{s}^{i})^{2}ds =\int_{0}^{t}R_{s}ds$,  by Levy's characterization of Brownian motion, $(W'_t)_{t\ge 0}$ given by 
$$
W'_t := \int_{0}^{t}\sum_{i=1}^{n} \frac{U_{s}^{i}dW_{s}^{i}}{\sqrt{R_{s}}}
$$ 
is a standard Brownian motion. 
Hence, $(R_{t})_{t\ge 0}$ is a weak solution to the CIR SDE \eqref{cir}. In particular, $\mathcal{L}(R_t)$ is the unique time-$t$ marginal distribution of any solution to the CIR SDE \eqref{cir} with the given parametrization. 

The conditions of part (a) implies that Theorem \ref{P1} can be applied to $U^{1},\dots,U^{n}$ jointly which gives
$$
\big(U^{1}_t,\dots,U^{n}_t\big)\stackrel{d}{\to}\sum_{j\in S}\delta_{U}(\{j\})\sqrt{V_j}\big(N_{1},\dots,N_{n}\big)
\quad\text{as }t\to\infty,
$$
where $U\indep \big(N_{1},\dots,N_{n}\big)$, $U\sim\pi$, $N_{1},\dots,N_{n}$ are i.i.d. standard normal and $U,N_{1},\dots,N_{n}$ independent of $V_j$ given in Theorem \ref{P1}. Consequently,
$$
R_{t}\stackrel{d}{\to} \sum_{j\in S}\delta_{U}(\{j\})\sqrt{V_{j}}\sum_{i=1}^{n}N_{i}^2\quad\text{as }t\to\infty.
$$
For parts (b) and (c), notice that 
\begin{align}
P\big[R_{t}\in A\big]&=
\sum_{j\in S}P_{ij}(0,t)P\Big[\sum_{i=1}^{n}(U_{t}^{i})^{2}\in A\mid X_{0}=i, X_{t}=j \Big].
\label{AppP1e1}
\end{align}
Note that similar to \eqref{tre14}, $
P\Big[\big((U_{t_{}}^{1})^{\frac{2}{\sqrt{t_{}}}},\dots,(U_{t_{}}^{n})^{\frac{2}{\sqrt{t_{}}}}\big)\in\cdot \mid X_{0}=i\Big]
=P\Big[\big((U_{t_{}}^{1})^{\frac{2}{\sqrt{t_{}}}},\dots,(U_{t_{}}^{n})^{\frac{2}{\sqrt{t_{}}}}\big)\in\cdot\Big]
$
and 
\begin{align}
\frac{1}{e^{^{-\sqrt{t_{}}2 E_{\pi}a(\cdot)}}}
\big((U_{t_{}}^{1})^{\frac{2}{\sqrt{t_{}}}},\dots,(U_{t_{}}^{n})^{\frac{2}{\sqrt{t_{}}}}\big)
\stackrel{d}{\to}
\exp\Big\{\frac{\sigma^{2}_{j}}{\sqrt{E|I_{j}|}}\big(N_{1},\dots,N_{n}\big)\Big\}
\quad\text{as } t\to\infty.\label{CIR3}
\end{align}
Write
\begin{align}
R_{t}=\sum_{i=1}^{n}(U_{t}^{i})^{2}=(U_{t}^{(n)})^{2}\Big[1+\sum_{i=1}^{n-1}\frac{(U_{t}^{(i)})^{2}}{(U_{t}^{(n)})^{2}}\Big],
\label{CIR4}
\end{align}
where, for each $t$, $U^{(n)}_t\ge \dots \ge U^{(1)}_t$ denote the ordered values. 
Boundedness of  
$$
1\le 1+\sum_{i=1}^{n-1}\frac{(U_{t}^{(i)})^{2}}{(U_{t}^{(n)})^{2}} \le n, 
$$
gives
\begin{align}
\Big[1+\sum_{i=1}^{n-1}\frac{(U_{t_{}}^{(i)})^{2}}{(U_{t_{}}^{(n)})^{2}}\Big]^{\frac{1}{\sqrt{t_{}}}}\stackrel{P}{\to} 1\quad\text{as }t \to\infty.\label{CIR5}
\end{align}
If $\mu$ denotes the weak limit in \eqref{CIR3}, then the discontinuity set $D_h$ of the mapping $(x_1,\dots,x_n)\mapsto \max_{1\le i\le n}x_i=:h(x_1,\dots,x_n)$ satisfies $\mu(D_h)=0$. Hence, the continuous mapping theorem applied to \eqref{CIR3} and $h$ gives
\begin{align}
\frac{1}{e^{^{-\sqrt{t_{}}2 E_{\pi}a(\cdot)}}}
(U_{t_{}}^{(n)})^{\frac{2}{\sqrt{t_{}}}}
\stackrel{d}{\to}
\exp\Big\{\frac{\sigma^{2}_{j}}{\sqrt{E|I_{j}|}}\max_{1\le i\le n}N_{i}\Big\}
\quad\text{as } t\to\infty.\label{CIR6}
\end{align}
Combining \eqref{CIR4}, \eqref{CIR5} and \eqref{CIR6} gives 
$$
\frac{1}{e^{^{-\sqrt{t_{}}2 E_{\pi}a(\cdot)}}}
R_{t_{}}^{\frac{1}{\sqrt{t_{}}}}
\stackrel{d}{\to}
\exp\Big\{\frac{\sigma^{2}_{j}}{\sqrt{E|I_{j}|}}\max_{1\le i\le n}N_{i}\Big\}
\quad\text{as } t\to\infty,
$$
which, upon taking the logarithm using the form in \eqref{AppP1e1} applying the limit argument used in Lemma \ref{lem1.5}, gives the statement in part (b) of the proposition to be proved.
The statement in part (c) is proved similarly.
\hfill $\square$
\end{proof}

\subsection{Proof of Proposition \ref{SIS}}
\begin{proof}
Using that $S_t=n-I_t$ and conditioning on the path of $X$, the dynamics of the infected population can be expressed as
$$
\frac{dI_t}{dt}=\beta(X_t)(n-I_t)I_t-\alpha(X_t)I_t.
$$
Dividing each side by $I^2_t$ the above dynamics is expressed as the following first order ODE for $I^{-1}$:
\begin{align*}
\frac{1}{I_t}=\frac{1}{I_0}e^{-\int_{0}^{t}\gamma(X_{s})ds}+\int_{0}^{t}e^{-\int_{s}^{t}\gamma(X_{r})dr}\beta(X_{s})ds.
\end{align*}
The result now follows immediately after using Corollaries \ref{Cor1} and \ref{Cor3}. For $E_{\pi}\gamma(\cdot)>0$, Corollary \ref{Cor1} provides weak convergence for the second term while the first term goes to $0$ as $t\to\infty$. 
For $E_{\pi}\gamma(\cdot)<0$, using Corollary \ref{Cor3}(a) gives 
\begin{align*}
P\Big[I_{t}\in \Big(e^{tE_{\pi}\gamma(\cdot) - b\sqrt{t}},e^{tE_{\pi}\gamma(\cdot) - a\sqrt{t}}\Big)\Big]
&=P\Big[-\frac{\log I_{t}}{\sqrt{t}}+\sqrt{t}E_{\pi}\gamma(\cdot) \in (a,b)\Big]\\
&\to\sum_{j\in S}\pi_{j}P\big[N\in\big(a\sigma_{j},b\sigma_{j}\big)\big] \quad\text{as }t\to\infty.
\end{align*}
For $E_{\pi}\gamma(\cdot)=0$, using Corollary \ref{Cor3}(b) gives
\begin{align*}
P\Big[I_{t}\in \big(e^{-b\sqrt{t}},e^{-a\sqrt{t}}\big)\Big]=P\Big[-\frac{\log I_{t}}{\sqrt{t}}\in (a,b)\Big]
\to\sum_{j\in S}\pi_{j}P\big[|N|\in (a\sigma_{j},b\sigma_{j})\big] \quad\text{as }t\to\infty.
\end{align*}
\hfill $\square$
\end{proof}

\section{Acknowledgements} The second author thanks Daniele Cappelletti and Carsten Wiuf for several illuminating discussions on a related problem, \cite{longTreaction2019}.
In particular, the idea used to prove Lemma 6.3 is inspired from Lemma 6.3 in \cite{longTreaction2019}.

\end{document}